\documentclass[12pt]{article}
\usepackage[english]{babel}
\usepackage[left=2.5cm,right=3cm,top=2.5cm,bottom=2.5cm,includeheadfoot]{geometry}

\usepackage{t1enc}
\usepackage[pdftex]{color,graphicx}
\usepackage[colorlinks=true]{hyperref} 
\usepackage{amsmath, amssymb}
\usepackage{amsthm,amsfonts}
\usepackage{eurosym}
\usepackage{enumerate}
\usepackage{eqnarray,subeqnarray}
\usepackage{subcaption}
\usepackage{graphicx}
\usepackage{latexsym}

\usepackage{todonotes} 
\newcounter{todocounter}

\usepackage{mathtools}                                                      
\mathtoolsset{showonlyrefs=true}                                    

\allowdisplaybreaks 
\makeatletter
\def\timenow{\@tempcnta\time
\@tempcntb\@tempcnta
\divide\@tempcntb60
\ifnum10>\@tempcntb0\fi\number\@tempcntb
:\multiply\@tempcntb60
\advance\@tempcnta-\@tempcntb
\ifnum10>\@tempcnta0\fi\number\@tempcnta}
\makeatother

\usepackage{etoolbox}
\makeatletter
\patchcmd{\@chapter}{#1}{#2}{}{} 
\patchcmd{\@chapter}{#1}{#2}{}{} 
\patchcmd{\@chapter}{#1}{#2}{}{}
\patchcmd{\@sect}{\fi#7}{\fi#8}{}{} 
\patchcmd{\@sect}{\fi#7}{\fi#8}{}{} 
\makeatother

\usepackage[utf8x]{inputenc}

\newtheorem{thm}{Theorem}[section] 
\newtheorem{prop}[thm]{Proposition}
\newtheorem{lemma}[thm]{Lemma}

\newtheorem{dfn}[thm]{Definition}
\newtheorem{properties}[thm]{Properties}
\theoremstyle{definition}
\newtheorem{exam}[thm]{Example}
\newtheorem{remark}[thm]{Remark}
\numberwithin{equation}{section}

\renewcommand\epsilon{\varepsilon}
\newcommand \esssup{\mbox{\upshape ess}\sup}
\newcommand \E{\mathbb{E}}
\newcommand \Esp[1]{\E\left[#1\right]}

\renewcommand \P{\mathbb{P}}
\newcommand \Prob[1]{{\P\left(#1\right)}}

\newcommand \dd{{\rm d}}

\newcommand \dx{\dd x}
\newcommand \dy{\dd y}
\newcommand \dz{\dd z}

\newcommand \N{\mathbb{N}}
\newcommand \R{\mathbb{R}}

\newcommand \cA{\mathcal{A}}
\newcommand \cE{\mathcal{E}}
\newcommand \cF{\mathcal{F}}
\newcommand \cG{\mathcal{G}}
\newcommand \cH{\mathcal{H}}
\newcommand \cN{\mathcal{N}}
\newcommand \cI{\mathcal{I}}
\newcommand \cJ{\mathcal{J}}
\newcommand \cL{\mathcal{L}}
\newcommand \cP{\mathcal{P}}

\newcommand \RareSetnew{{E_{0}}}

\newcommand \1{\mathbf{1}}
\newcommand \PZj{\P_{Z_j}}

\newcommand \emzon{A_{Z_{1:n}}}
\newcommand \emzJ{A_{Z_{J}}}
\newcommand \emzsmaJ{A_{z_{_{J}}}}
\newcommand\emxon{A_{X_{1:n}}}
\newcommand\emxyon{A_{(X,Y)_{1:n}}}

\newcommand\emxJ{A_{X_{J}}}
\newcommand \emzJsta{A_{Z_{J}^{*}}}

\newcommand \emxyJsta{A_{(X,Y)_{J}^{*}}}
\newcommand \mmzon{\mu_{Z_{1:n}}}
\newcommand \mmxon{\mu_{X_{1:n}}}

\newcommand \mmxyon{\mu_{(X,Y)_{1:n}}}

\newcommand \mmxyJsta{\mu_{(X,Y)_{J}^{*}}}

\newcommand \mmzJ{\mu_{Z_{J}}}
\newcommand \mmxJ{\mu_{X_{J}}}
\newcommand{\regk}{\Phi_k} 

\newcommand{\reg}{{ \Phi}_{1:n}}
\newcommand{\regave}{\Phi_{n}}

\newcommand{\hatreg}{\widehat \Phi_n}

\newcommand \bgf{{g}_{f,1:n}}

\newcommand \bgfon{{g}_{f,1:n}}
\newcommand \bgfJ{{g}_{f,J}}

\newcommand \frackn{\frac{1}{n}\sum_{k=1}^{n}}

\usepackage[framemethod=tikz]{mdframed}
\usetikzlibrary{decorations.pathmorphing,calc}
\newcommand\wavydecor{%
    \draw[decoration={coil,aspect=0.1,segment length=5pt,amplitude=1.0pt},decorate,line width=1.5pt,black]
      (O|-P) -- (O);
}
\newmdenv[
hidealllines=true,
innerleftmargin=10pt,
innerrightmargin=0pt,
innertopmargin=0pt,
innerbottommargin=0pt,
leftmargin=-10pt,
skipabove=.5\baselineskip,
skipbelow=.5\baselineskip,
singleextra={\wavydecor},
firstextra={\wavydecor},
secondextra={\wavydecor},
middleextra={\wavydecor}
]{done}

\newcommand \nch[1]{{\color{black}#1}}

\usepackage{fancyhdr}

\title{\nch{Generalization bounds for  nonparametric regression  with $\beta-$mixing samples}} 
\author{David Barrera  \thanks{Corresponding author.} \thanks{Email: {\tt juandavid.barreracano@epfl.ch}. CMAP, \'{E}cole Polytechnique, Route de Saclay, 91128 Palaiseau cedex, France and SDS, \'{E}cole Polytechnique F\'{e}d\'{e}rale de Lausanne. EPFL SB MATH
MA C2 647 (B\^{a}timent MA) 
Station 8
CH-1015 Lausanne,
Switzerland. Supported in 2019 by the Chaire March\'{e}s en Mutation, F\'{e}d\'{e}ration Fran\c{c}aise Bancaire and by the Institut Louis Bachelier.} 
\and
Emmanuel Gobet \thanks{Email: {\tt emmanuel.gobet@polytechnique.edu}. CMAP, Ecole Polytechnique, Route de Saclay, 91128 Palaiseau cedex, France. The authors research is  part of the Chair {\it Financial Risks} of the {\it Risk Foundation} and the {\it  Finance for Energy Market Research Centre}. This research also benefited from the support of the Chair {\it Stress Test, RISK Management and Financial Steering}, led by the French \'{E}cole Polytechnique and its Foundation and sponsored by BNP Paribas.}}

\date{}

\begin{document} 
 \maketitle
\thispagestyle{empty}
\begin{abstract}
In this paper we present a series of results that permit to extend in a direct manner uniform deviation inequalities of the empirical process from the independent to the dependent case characterizing the additional error in terms of $\beta-$mixing coefficients associated to the training sample. We then apply these results to some previously obtained inequalities for independent samples associated to the deviation of the least-squared error in nonparametric regression to derive corresponding generalization bounds for regression schemes in which the training  sample  may not be independent. 

These  results  provide  a framework to analyze  the error associated to regression schemes whose training sample comes from a large class of  $\beta-$mixing sequences, including geometrically ergodic Markov samples, using only the independent case. More generally, they permit a meaningful extension of the Vapnik-Chervonenkis and similar theories for independent training samples to this class of $\beta-$mixing samples. 
 \end{abstract}
\newpage
\setcounter{tocdepth}{2}
{
  \hypersetup{linkcolor=black}
 \tableofcontents
}
\section{Introduction and background}
\label{secint}
This paper is a continuation of \cite{bargob}, where we addressed the problem of  studying the error associated to a least-squares regression scheme in the nonparametric, distribution--free setting assuming that the training sample is independent.

\subsection{The problem} 

Let $n\in \N:=\{1,2,\dots\}$ be a natural number (the ``sample size'')
, let the ``training sample'' of ``explanatory inputs'' $X_{k}$ and ``responses'' $Y_{k}$
\begin{align}
D_{n}:=((X_{k},Y_{k}))_{k\in \{1,\dots,n\}}
\end{align}
 be a (not necessarily i.i.d.) random 
sequence in $S\times \R$, where $S$ is a Polish space, defined on the probability space $(\Omega, \cE, \P)$ , and let $\cF_{n}$ be a family of Borel-measurable functions $S\to \R$ (the ``space of hypotheses''). For $k\in \{1,\dots, n\}$, denote by $\P_{X_{k}}$ [respectively $\P_{(X_{k}, Y_{k})}$ ] the law of $X_{k}$ [respectively $(X_{k},Y_{k})$], assume that $Y_{k}\in L^{2}_{\P_{X_{k}}}$, and let $\regk:S\to \R$ be a version of the conditional expectation of $Y_{k}$ given $X_{k}$, thus
\begin{align}
\label{equfirintregk}
\regk(X_{k})=\Esp{Y_{k}|X_{k}}, \,\,\, \P-a.s.
\end{align}

Given such $(n,D_{n},\cF_{n})$, a natural candidate to a ``simultaneous'' estimator within $\cF_{n}$ of the regression functions $\regk$ 
 is the empirical regression function $\hatreg$ defined as a solution to the {\it least-squares regression problem} 
\begin{align}
\label{equdefhatregfir}
\hatreg\in \arg\min_{f\in \cF_{n}}\frackn|f(X_{k})-Y_{k}|^{2}.
\end{align}
Indeed, by the orthogonal decomposition 
\begin{align}
\Esp{|Y_{k}-f(X_{k})|^{2}}=\Esp{|Y_{k}-\regk(X_{k})|^{2}}+\Esp{|f(X_{k})-\regk(X_{k})|^{2}},
\end{align}
the solutions $\regave^{*}$ to the problem
\begin{align}
\label{equdefregavefir}
\regave^{*}\in \arg\min_{f\in\cF_{n}}\frackn\Esp{|f(X_{k})-\regk(X_{k})|^{2}}
\end{align}
 are the same as those to the problem 
 \begin{align}
\label{equdefhatregsim}
\regave^{*}\in \arg\min_{f\in \cF_{n}}\frackn\Esp{|f(X_{k})-Y_{k}|^{2}},
\end{align}
from where it follows that \eqref{equdefhatregfir} and \eqref{equdefregavefir} are approximately the same problem  {\it provided that the deviations of the random variables inside the $\arg\min$ in \eqref{equdefhatregfir} from their expectations inside the $\arg\min$ of \eqref{equdefhatregsim} are {\upshape(in some appropriate sense)} ``negligible''  uniformly in $\cF_{n}$
}. 

In this context, the purpose of \cite{bargob}\footnote{Where we assumed $S=\R^{d}$, which is nonetheless largely irrelevant for the arguments.} was roughly speaking to show that, {\it when  $D_{n}$ is a sequence of independent random variables}, such deviations can be properly controlled {provided a control on the complexity of $\cF_{n}$}\footnote{As measured typically by uniform entropy estimates, see Definition \ref{defl1entest}.} {and a uniform bound of the response variables $Y_{k}$}, and to describe some of the consequences of these controls for the problem of (weak and strong) rates and consistency, including the case where the response sequence $(Y_{k})_{k}$ is not bounded. The innovation in  \cite{bargob} with respect to the classical i.i.d. case is, therefore, in the non-stationarity of $D_{n}$.

{In continuation with this, we aim here at deriving some bounds for the probability of uniform deviations like
\begin{align}
\label{equgenunidev}
\Prob{\sup_{(g_1,\dots,g_n)\in 
\cG_{1,\dots,n}
} \frac 1 n\sum_{j=1}^n \left(a g_j(X_j,Y_{j}) + b\int g_j(x,y)\P_{X_j,Y_{j}}(\dx \,\dy)\right)\geq t}
\end{align}
when the training data $D_{n}$ is not necessarily stationary, {\it nor} independent, but  satisfies some $\beta-$mixing properties (particularly those in Definitions \ref{defsubbetmix} and \ref{defsubpolmix}).  Here $a,b,t$ are scalar,  $\cG_{1,\dots,n}$ is a family of vectors $(g_{1},\dots, g_{n})$ whose entries are measurable functions $S\times \R \to \R$, and the complexity of $\cG_{1,\dots,n}$ is controlled in the same ways as in \cite{bargob}.
}

We will show here how to ``lift'' the deviation inequalities in \cite{bargob} from the independent to the dependent case using decoupling techniques associated to the $\beta-$mixing coefficients of the training sample, and we will generalize some of the  consequences for weak consistency and bounds on weak errors obtained in \cite{bargob} for independent training samples using these ideas. When interpreted in the Markovian setting, these results provide error rates and consistency theorems for least-squares regression schemes under important ergodicity conditions on  $D_{n}$. See for instance \cite{tuotwe}, \cite{jarrob}, 
 \cite{douforgui},  and the references therein.

\subsection{Motivation}
Our study is motivated {in particular} by the following application.  In \cite{fort:gobe:moul:16}, the authors investigate the  numerical computation of the mean of a function of a conditional
expectation in a rare-event regime, which takes the form

\begin{equation}
\label{eq:pb:base}
\cI := \Esp{f(
\nch{\tilde{X}},\Esp{Y\vert \tilde{X}
})\vert \nch{\tilde{X}\in \RareSetnew}},
\end{equation}
where \nch{$\tilde{X}$} and $Y$ are random variables, and the event \nch{$\RareSetnew\in \cE$} is rare (i.e. $\P(\nch{\tilde{X}\in \RareSetnew
})$ small). This problem is prominent in financial/actuarial risk management  when, as often, one has to deal with future risk exposure (modelled by $\Esp{Y \vert \tilde{X}
}
=:\Phi(\tilde{X})$) in extreme configurations (described by the set $\RareSetnew$). The above can be rewritten as $\cI=\Esp{f(X,\Esp{Y\vert X})}$ where $X$ \nch{has} the conditional distribution of $\tilde{X}$ given \nch{$\{\tilde{X}\in \RareSetnew \}$}. The computational strategy developed in \cite{fort:gobe:moul:16} consists in sampling $n$ times $(X,Y)$, computing the empirical regression function $\hat \Phi_n(x)\approx\Esp{Y\vert {X}=x}$ with these data, and averaging out the results over the \nch{explanatory} sample $X_1,\dots, X_n$.  One specific issue is that, \nch{$\RareSetnew$}  being rare, naive i.i.d. sampling of $X$ (with acceptance-rejection on \nch{$\RareSetnew$} ) is quite inefficient and one has to resort 
\nch{to} a MCMC technique. The new $X_1,\dots, X_n$ are thus not independent, nor stationary, but they fulfill some good $\beta-$mixing properties to ensure the approximation with respect to the (target) distribution of $X$. The convergence analysis is developed in \cite{fort:gobe:moul:16} and a upper bound  on the Mean Square empirical norm 
$$ \Esp{ \frac{1}{n} \sum_{j=1}^n \left( \hat \Phi_n(X_j) - \Phi(X_j)\right)^2 }$$
is derived.  

Using the current results of this work, we will be able to extend the scope of validity of the error analysis \nch{in \cite{fort:gobe:moul:16}} in two directions: first, allowing the  functions class for computing $\hat \Phi$ to be \nch{more} general (and not only \nch{a} linear space as in \cite{fort:gobe:moul:16}),  including   neural networks for instance; second, estimating the out-of sample error (as opposed to the in-sample error -- aka empirical error).

\subsection{Contributions of this paper}
The results in this paper contribute to the existing literature mainly in two directions,

\begin{enumerate}
\item {\it A systematic presentation of the  ``lifting'' of uniform deviation inequalities  via Berbee's lemma.} This occupies Section \ref{secbribetmixind}, whose main results are  Theorems \ref{theconbet} and Proposition \ref{prop_corsemabsexp}. 

While the main purpose of this part of the paper is to permit a smooth and clear transition from some of the results under independence treated in \cite{bargob} to the corresponding generalizations to dependence with   $\beta-$mixing errors (achieved in Section \ref{secapp}), we aimed to present the results in this section in a manner that makes clear how these ideas go far beyond in generality than the kind of applications for which they are developed here.  In this sense, we hope that they might serve as a useful reference for other works in which deviation inequalities for nonindependent sequences are sought for, provided that their independent counterparts are known or clearly obtainable.

\item {\it Weak rates and consistency theorems for least-squares regression schemes with nonindependent training samples.} This part, developed in Section \ref{secapp},  consists in an application of the results from Section \ref{secbribetmixind}  to some of the results and proofs in \cite{bargob}. 

The conclusions obtained (see for instance Theorems \ref{the22bargobbetmix} and \ref{theweaerrestbet}) allow us to see how  some the estimates obtained in \cite{bargob} for independent samples generalize to estimates for dependent samples\footnote{ It is important to emphasize that, by reasons of space, in this process of generalizing we  did not exhaust all the results available in \cite{bargob}. 
The arguments for those treated here indicate how to extend the ones left aside.} via the results from Section \ref{secbribetmixind}.  These estimates are   meaningful for a class of training samples with a kind of ``superlinear $\beta-$mixing rate'' (see \eqref{equbetmixo1oven} and  \eqref{equdefquabetmixfas}), providing in particular non-parametric, distribution-free estimates for geometrically ergodic Markovian training samples. 
\end{enumerate}

\subsection{Background literature}
 Concentration and deviation inequalities for nonindependent samples constitute a topic of considerable research, in particular due to the importance of the Markovian case at the level of applications.   

We start by mentioning  \cite{ren:moji:10},  which uses basically  the same coupling ideas developed in the present paper\footnote{Our developments were indeed considerably inspired by the argument in \cite{ren:moji:10}.} to extend some of  the inequalities in \cite{gyor:kohl:krzy:walk:02} for i.i.d. samples to the stationary $\beta-$mixing case and to describe the respective consequences for estimates of weak errors of least--squares regression schemes, including some penalisations. 
Our results give estimates that cover in the nonstationary case the corresponding estimates in \cite{ren:moji:10} with a very significant improvement on the constants involved. These gains come in   part from the work developed in \cite{bargob}.

We also mention \cite{adam:08} (see also references therein). 
This paper presents first a  deviation estimate (\cite[Theorem 4]{adam:08}) for independent samples under the assumptions that the functions in the space of hypotheses are centered with respect to the marginal laws of the sample and satisfy some bounds in terms of Orlicz norms,  and then develops similar estimates (\cite[Theorems 6 and 7]{adam:08}) for uniformly bounded Markov samples under a certain  ``minorization condition'' (\cite[Section 3.1]{adam:08}).  In contrast with our results, the estimates for independent samples in \cite{adam:08} cover cases in which the family of hypotheses is {not} uniformly bounded. Our estimates, on the other side, do not require the centering of the hypotheses with respect to the marginal laws in the independent case, and give rates for any exponentially $\beta-$mixing sequence of samples even if it is not Markovian, covering in particular the geometrically ergodic Markov chains in \cite{adam:08}.  We point out also that our applications (mainly Theorem \ref{the22bargobbetmix}) give bounds which are upper estimates on the probability of {\it some} individual large  deviation of the empirical processes parametrized by the family of hypotheses    from its corresponding mean, whereas the uniform estimates in \cite{adam:08} (\cite[Theorems 4 and  7]{adam:08}) are rather estimates on the probability of  deviations of the  supremum of these empirical process from its mean: we will refer to   these as ``tail estimates'' in the rest of this section.

In \cite{KM16}, a coupling argument  similar to the one in the present paper  is used to address the problem of generalisation bounds for unspecified loss functions of regression algorithms  in term of Rademacher complexities and $\beta-$mixing coefficients associated to  dependences in the training sample, in a setting whose generality is approximately the same as that in our Section \ref{secbribetmixind}. The  argument in \cite{KM16}, which  proceeds via McDiarmid's inequality (see footnote \ref{foonotmcdia} below), has the advantage of  simplicity and generality compared to ours, but the rate obtained (roughly speaking ${1}/{\sqrt{n}}$ where $n$ is the sample size) is suboptimal for the (square) loss function considered in our paper (we obtain roughly the rate $\log n /n$ in our analysis). For further comparison, notice  again that our analysis does {\it not} proceed via tail estimates (see the comparison with \cite{adam:08} before), and that we also cover the case of  hypotheses  depending on the index of the sample (the ``time'').  

At a more ergodic theoretical level, let us mention the result in \cite{NobDem93}, where it is proved that the uniform convergence of averages holds for $\beta-$mixing samples (with stationary marginals) provided that it holds for i.i.d. samples with the same marginals when the class of functions in consideration has finite Vapnik-Chervonenkis (VC) dimension (as defined here in Example \ref{exasaushe}). Our paper can in part be considered a  continuation of this story towards the investigation of rates of convergence, with more freedom in the independence assumption but with restrictions on the speed of mixing. 

Let us comment briefly on the related research about these rates. Rates of uniform convergence to zero for the centered averages were for instance investigated  in  \cite{YUK86} (see also references therein), where the sample sequence is a $\phi-$mixing (and therefore $\beta-$mixing)  process whose $\phi-$mixing coefficients  satisfy certain growth conditions, and where the class of hypotheses is assumed to satisfy some ``weak metric entropy'' conditions and some controls on the associated maximal variance (see \cite[Conditions (1.1)--(1.4), (1.6), and (1.8)--(1.10)]{YUK86}). Another instance of this story, closer to our paper, is  \cite{Yu94}, which works under a general framework and via techniques that are quite similar to the ones here. It considers a case in which the sample sequence is $\beta-$mixing under   a  decay of the $\beta-$mixing coefficients that can be slower than ours, and it is also an interesting source of additional references. The results in \cite{Yu94} complement our results in so far as \cite{Yu94} considers slower mixing rates, and are complemented by our results in so far as \cite{Yu94} 
 relies on the assumption  of stationary samples and time--independent spaces of hypotheses, which we dispense with here.

Like our own, many of the aforementioned papers proceed via comparisons with the corresponding results for the independent case and   clever bounds on the additional error induced by dependence. The argument for the independent case  typically depends on estimates of probabilities like  \eqref{equgenunidev} when 
\begin{align}
\label{equatodev}
\cG_{1,\dots, n}=\{(g_{1},\dots,g_{n})\}
\end{align}
 consists of a single point (``atomic estimates'') and the  training sequence $D_{n}$ is independent, from where the uniform estimates (for more general $\cG_{1,\dots,n}$) follow via finitely many applications of the atomic estimates using, for instance,  ``symmetrisation'', ``chaining'', and estimates of covering or bracketing numbers  (``entropy estimates''). See for instance \cite{P90} for an introduction to these ideas.

These  estimates have nonetheless been studied ``directly'' under classical dependence conditions in several works.
The arguments in \cite{YUK86}, for instance, depend on a result (\cite[Lemma 2.1]{YUK86}) which is an extension to the $\phi-$mixing case of Bernstein's inequality. 

But the developments in this directions have continued until recent years. One example is \cite{merl:peli:rio:09} (see also references therein), whose results (\cite[Theorems 1 and 2]{merl:peli:rio:09}) imply that, if each $g_{j}$ in \eqref{equatodev} is bounded and $a=-b=1$ (centered case), and if the $\alpha-$mixing coefficients associated to the sample sequence  decay exponentially (\cite[Condition (1.3)]{merl:peli:rio:09}\footnote{This condition is weaker than \eqref{equdefexpmix} below for $\gamma=1$, but we remind  that the estimates in \cite{merl:peli:rio:09} are not uniform.}), then a   Bernstein--type inequality bound holds (under  \eqref{equatodev}) at the right--hand side of \eqref{equgenunidev}.     A second and final one is \cite{dede:gouz:15}, where it is shown that, in the context of irreducible and aperiodic Markov chains, the assumption of geometric ergodicity is equivalent to the satisfaction of McDiarmid-type inequalities for separately bounded functionals of the observables\footnote{\label{foonotmcdia} If $(X_{k})_{k}$ is the Markov chain in consideration, this amounts  to the satisfaction of estimates of the type
\begin{align}
\Prob{|K(X_{1},\dots, X_{n})-\Esp{K(X_{1},\dots, X_{n})}|>t}\leq C_{1}\exp\left({-C_{2}t^{2}}{ /}\sum_{k=1}^{n}L_{k}^{2}\right)
\end{align}
where $K:\R^{n}\to \R$ is any (Borel-measurable) function such that $x\mapsto K(x_{1},\dots,x_{k-1},x,x_{k+1},\dots,x_{n})$  is  bounded by $L_{k}>0$ when $x_{1},\dots,x_{k-1},x_{k+1},\dots,x_{n}$ is fixed, for every $k\in 1,\dots,n$. Notice in particular that this  covers tail estimates like those in \cite{adam:08} and \cite{KM16}  when the entries of $\cG_{1,\dots,n}$ are uniformly bounded. For potential comparisons of \cite{dede:gouz:15}  with our results see again the comparison with \cite{adam:08} and \cite{KM16} above. 
} (\cite[Theorem 2 and Remark 4]{dede:gouz:15}). One of the conclusions in \cite{dede:gouz:15} is that, for the small set specified in \cite[Definition 1]{dede:gouz:15}, these inequalities hold (also) under the conditional law at every starting point in such set and for the deviations of the expectation with respect to such conditional law. 

For the case of suprema of partial sums, the results explained in Section \ref{secabslifdevine}  are comparable with those in \cite{dede:gouz:15}:  they give analogous consequences  for the probability of large deviations\footnote{As opposed, again, to the tail estimates that follow from \cite{dede:gouz:15}.} which rely only on the rate of decay of the $\beta-$mixing coefficients associated to the underlying sequence and on the corresponding estimates from the independent case. These estimates admit therefore as a special case that in which the training sample comes from a Markov chain as those in \cite{dede:gouz:15}.

\paragraph{Organization of the paper.} The rest of the paper is organized as follows:  we begin Section \ref{secbribetmixind} by introducing some notational conventions that will be used in the forthcoming pages. We explain next, also in Section \ref{secbribetmixind}, how to transport uniform deviation inequalities from the independent to the dependent case estimating the additional error via the $\beta-$mixing coefficients. Section \ref{secapp} presents some applications to problems in nonparametric least--squares regression under dependent training samples, in continuity with some of the independent-case considerations in \cite{bargob}. 

\section{Bridge between $\beta-$mixing and independent sequences}
\label{secbribetmixind}
Our strategy for deriving concentration-of-measure inequalities for dependent sequences is to leverage on decoupling techniques and  deviation inequalities for independent sequences (as those of  \cite{bargob}). These inequalities with dependent sequences will take the form of Lemma \ref{thedisfungen} and Proposition  \ref{prop_corsemabsexp}, which constitute the main result of this section. The derivation is made in several steps.

\subsection{Notation and conventions}
\label{secnot}
The following conventions will be used in this paper:
\begin{itemize}
\item We depart from a probability space $(\Omega,\nch{\cE},\P)$ supporting all the  random variables that will appear in our statements and proofs (the existence of this space can be verified {\it a posteriori}).

\item \nch{We denote by $\N=\{1,2,\dots,\}$  the set of positive integers.}

\item For $k,n\in \N$, we will sometimes denote $k:n:=\{k,\dots,n\}$ ($k:n:=\emptyset$ if $k>n$), and we use the notation $c_{1:n}$ for a sequence ($n-$tuple) of elements $(c_{1},\dots,c_{n})$.

\item More generally, given a subset $J\subset \N$, $c_{J}:=(c_{j})_{j\in J}$ denotes a sequence indexed by $J$, which we will call a {\it $J-$tuple}. {The cardinality of $J$ is denoted by $|J|$. }If $c_{J}=(c_{j})_{j\in J}$ is given and $J'\subset J$,  we will denote {\it the projection of $c_{J}$ onto the $J'$ coordinates} by $c_{J'}$\footnote{Of course, we will be careful to use properly the notation to avoid confusions: in no place we will for instance denote two different tuples as $c_{J}$ and $c_{J'}$, except  if their entries with index in $J\cap J'$ are equal.}. Thus for $c_{J}:=(c_{j})_{j\in J}$, 
\begin{align}
\label{equproaj}
c_{J'}=(c_{j})_{j\in J'}.
\end{align}
\item For a subset $J\subset \N$ and a family of sets  $\{C_{j}\}_{j\in J}$  indexed by $J$, we use the notation 
\begin{align}
\label{equdefpro}
C^\otimes _{J}:=\{c_{J}=(c_{j})_{j\in J} |\, \forall j \in J: c_{j}\in C_{j}\}
\end{align}
for {\it the product of the $C_{j}$'s}\footnote{The same care will be taken to avoid confusion here: we will always use the same character (here ``$C_{\cdot}$'') for the sets involved in the product.}. 

\item  Sometimes\footnote{Especially for function hypotheses, see for instance \eqref{equdefcalgf} below.} we will deal with {\it sets $\mathcal{F}_{J}$ of $J-$tuples} which are not necessarily a product of sets. In {all} of these cases the indexing set (i.e., $J$) of the elements of $\cF_{J}$ will be indicated in the notation. In analogy with \eqref{equproaj}, given such $\cF_{J}$ and $J'\subset J$, $\cF_{J'}$ denotes the projection of $\cF_{J}$ into the $J'$ coordinates
\begin{align}
\label{equdeffjpri}
\cF_{J'}:=\{f_{J'}:f_{J}\in \cF_{J}\},
\end{align}
where each $f_{J'}$ is given by \eqref{equproaj}. Thus for instance, for the set in \eqref{equdefpro} and $J'\subset J$, we have $(C^\otimes_{J})_{_{J'}}=C^\otimes_{J'}.$

\item We reserve the character $S$ for Polish spaces with variations from taking products as in the above,  and we  will usually denote by $Z$ a generic random vector in  $S$ with compatible variations when $S$ is a product space. Thus $Z_{J}$ typically denotes a random element of a product space $S^\otimes_{J}$. This is, $Z_{J}=(Z_{j})_{j\in J}$ with $Z_{j}:\Omega\to S_{j}$ $\cE-$measurable and $S_{j}$ a Polish space.

 \item If $Z$ is a random element of $S$ and $B$ a Borel set of $S$, we use the standard notation 
$ \{Z\in B\}:=\{\omega\in \Omega: Z(\omega)\in B\}$
for the preimage of $B$ (which is a set in $\cE$). We use similarly the standard notation $\P_{Z}$ for the law of $Z$: given a Borel set $B\subset S$, 
\begin{align}
\P_{Z}(B):=\P(\{Z\in B\}).
\end{align}

\item For a Polish space $S$, $\mathcal{L}_{S}$ denotes the space of Borel-measurable functions $S\to \R$. If $\{S_{j}\}_{j\in J}$  ($J\subset \N$) are Polish spaces, set $\mathcal{L}^{\otimes_{J}}_{S}:=\Pi_{j\in J}\cL_{S_{j}}$.  A subset of $\mathcal{L}^{\otimes_{J}}_{S}$ will be  called a {\it sequential family of functions compatible with $S_{J}^{\otimes}$},  or simply a {\it sequential family of functions} when there is no ambiguity for $\{S_{j}\}_{j\in J}$. \nch{One relevant example is the sequential family of functions \begin{align}
\label{equdefcalgf}
\cG_{\cF,1:n}:=\{g_{f,1:n}:f\in \cF\}
\end{align} 
in \eqref{equdefgfk}.}

\item When needed, we will operate with sequential families of functios in a componentwise manner, thus given $f_{J}=(f_{j})_{j},f_{J}'=(f_{j}')_{j}$ in $\cF_{J}$, where $\cF_{J}$ is a sequential family of functions, $f_{J}+f_{J'}=(f_{j}+f_{j}')_{j}$, $f_{J}f_{J}':=(f_{j}f_{j}')_{j}$, $\nch{|f_{J}|}:=(|f_{j}|)_{j}$, and so on.
\item A couple $(Z_{J},\cF_{J})$ where $Z_{J}$ is a random element of $S^\otimes_{J}$ and $\cF_{J}$ is a sequential family of functions compatible with $S^{\otimes}_{J}$ is called a {\it composable pair}. {In this definition, the reference to $S^{\otimes}_{J}$ is implicit and omitted for the sake of convenience.} Notice that if $(Z_{J},\cF_{J})$ is a composable pair and $J'\subset J$, then $(Z_{J'},\cF_{J'})$ is a composable pair.

\item   The {\it empirical mean} and the {\it average mean} associated to the {composable pair $(Z_J,f_{J})$},    denoted respectively by $\emzJ f_{J}$ and $\mmzJ f_{J}$ are defined, {for nonempty finite $J$}, as 
  \begin{align}
\label{equdefavemea}
\emzJ f_{J}:=\frac{1}{|J|}\sum_{j\in J} f_{j}(Z_{j}), & \qquad \mmzJ f_{J} 
=\frac{1}{|J|}\sum_{j\in J}\int_{S_{j}} f_{j}(z)\,\PZj(\dz). 
 \end{align}
(the second average is defined only for those $f_{J}$ where it makes sense, including the possible value $\infty$).  With this convention, we will use the short notation 
$$(a\emzJ+b\mmzJ) f_{J}:=a\emzJ f_J+b\mmzJ f_{J}= \frac{a}{|J|}\sum_{j\in J} f_{j}(Z_{j})+ \frac{b}{|J|}\sum_{j\in J}\int_{S_{j}} f_{j}(z)\,\PZj(\dz),$$ for any real constants $a,b$.
\item When convenient, we identify a function $f$ with the {\it constant} sequence of functions \nch{$(f_{j})_{j\in J}$ where $f_{j}=f$ for all $j\in J$}, which together with the above permits, for instance,  an unambiguous interpretation of the object ``$\mmzJ f$''.

\end{itemize}

\subsection{``Union bound'' for deviations of averages 
}
\sectionmark{Union bound for deviations}
We begin with the following elementary lemma, which shows that estimates on the distribution function associated to suprema of (generally non--centered) empirical means can be obtained from corresponding estimates on the empirical means over the indexes in a partition of the set  $\{1,\dots,n\}$.

\begin{lemma}[``Union bound'' for deviation of averages]
\label{thedisfungen}
Let \nch{$n\in \N$,}  let $\cJ$ be a partition (by nonempty subsets) of $\{1,\dots,n\}$ and let $(Z_{1:n},\cG_{1:n})$ be a composable pair.  
Then for every $(a,b,t)\in \R^{3}$
\begin{align}
\Prob{\sup_{{g_{1:n}}\in\cG_{1:n}}(a \emzon+b\mmzon) g_{1:n} \geq t}
\leq \sum_{J\in \cJ}\Prob{\sup_{{g_{J}}\in\cG_{J}}(a \emzJ+b\mmzJ) g_{J}\geq t}.
\label{equbousumpar}
\end{align} 
\end{lemma}

\begin{proof}
The proof is easy: for every $J\in \cJ$, denote  $\gamma_{J}:=|J|/n$. Notice  that $\sum_{J\in\cJ}\gamma_{J}=1$.  With this,  \eqref{equbousumpar} is an immediate consequence of the subadditivity of the supremum, linearity, and the union bound:
\begin{align}
\Big\{\sup_{{ g_{1:n}}\in\cG_{1:n}}(a \emzon+b\mmzon) g_{1:n} \geq t\Big\}
=\Big\{\sup_{{g_{1:n}}\in\cG_{1:n}} \sum_{J\in \cJ}\gamma_{J}(a\emzJ+b\mmzJ) g_{J} \geq \sum_{J\in \cJ}\gamma_{J} t\Big\}\\
\subset \Big\{\sum_{J\in \cJ} \gamma_{J}\sup_{{ g_{J}}\in\cG_{J}}(a\emzJ+b\mmzJ) g_{J}\geq \sum_{J\in \cJ}\gamma_{J} t\Big\}
\subset \bigcup_{J\in\cJ}\Big\{\sup_{{ g_{J}}\in\cG_{J}} (a\emzJ+b\mmzJ) g_{J}\geq  t\Big\}.
\end{align}
\end{proof}

{The above lemma shows that if we can find appropriate subsampling partition $\cJ$ for which we have an exponential (for instance) inequality for the deviation probability, the same type of inequality holds for the full sample   $\{1,\dots,n\}$. The construction of the partition $\cJ$ will be made using the $\beta-$mixing properties of the sequence $Z_{1:n}$, which is now discussed.}

\nch{\begin{remark}[Generalization under a convex-like estimate]
\label{remgenestfamfun}
Lemma \ref{thedisfungen} can clearly be extended to any family of (Borel-measurable) functionals $\{K_{J}\}_{J\subset \N}$, $K_{J}:S_{J}^{\otimes}\to \R$ with the property that for every disjoint family $\{J_{1},\dots ,J_{r}\}\subset 2^{\N}$ and some nonnegative $\gamma_{1},\dots,\gamma_{r}$ with $\sum_{k}\gamma_{k}=1$, $K_{J}(Z_{J})\leq \sum_{k}\gamma_{k}K_{J_{k}}(Z_{J_{k}})$, $\P-$a.s., where $J:=\cup_{k} J_{k}$. For such a family one has the inequality
\begin{align}
\Prob{K_{J}(Z_{J})\geq t}\leq \sum_{k=1}^{r}\Prob{K_{J_{k}}(Z_{J_{k}})\geq t},
\end{align}
for every $t\in \R$, every $J\subset \N$, and every partition $J_{1},\dots, J_{r}$ of $J$. See also Remark \ref{remvergenfamfunlif} below.
\end{remark}}

\subsection{The $\beta-$mixing coefficients}

In this section, we introduce some  facts about $\beta-$mixing coefficients that will be useful later. {For an account on mixing properties, we refer the reader to \cite{DDLLLP07, douk:12, douc:moul:prio:soul:17}.}
\subsubsection{Basic definitions and properties}

\begin{dfn}[$\beta-$mixing coefficients]
\label{defbetmixcoe}
Let $\cE_{1}$ and $\cE_{2}$ be two sub-sigma algebras  of $\cE$. The $\beta-$mixing coefficient $\beta(\cE_{1},\cE_{2})$ between $\cE_{1}$ and $\cE_{2}$ is defined as
\nch{ \begin{equation}
\label{betmixdef}
\beta(\cE_{1},\cE_{2}):=\Esp{\esssup_{E_{1}\in\cE_{1}}|\P({E}_{1})-\P[E_{1}|\cE_{2}]|}.
\end{equation}
}
\end{dfn}
\nch{For a definition of the essential supremum, ``$\esssup$'', of a family of random variables, see \cite[Proposition VI-1-1]{NEV75}. It follows in particular that there exists a countable family $\{E_{1,n}\}_{n}\subset \cE_{1}$ such that
\begin{align}
\label{equdefbetcousup}
\beta(\cE_{1},\cE_{2})=\Esp{\sup_{n}|\P(E_{1,n})-\P[E_{1,n}|\cE_{2}]|}.
\end{align}}
\begin{remark}[A characterization. Properties.]
\label{remprobetmix} 
\nch{
If $\{E_{1,n}\}_{n}$ is the family in \eqref{equdefbetcousup} and $\cE_{2}$ is countably generated, then
\begin{align}
\beta(\cE_{1},\cE_{2})&=\Esp{\sup_{n}|\P E_{1,n} -\Prob{ E_{1,n}|\cE_{2}}|}\\
&=\frac{1}{2}\sup_{(P_{{1}},P_{{2}})\in \cP_{\cE_{1}}\times \cP_{\cE_{2}}}\sum_{(E_{1}',E_{2}')\in{P_{{1}}}\times {P_{{2}}}}|\Prob{E_{1}'}\Prob{E_{2}'}-\Prob{E_{1}'\cap E_{2}'}|,
\label{equcharbetmixcoefuse}
\end{align} 
where $\cP_{\cE_{k}}$ ($k=1,2$) denotes the family of finite partitions of $\Omega$ by $\cE_{k}-$sets\footnote{This can be seen for instance by noticing that there exist increasing families of finite fields $\{\cE_{j,k}\}_{k}$ ($j=1,2$) with $\cup_{k}\cE_{j,k}\subset \cE_{j}$  such that 
\begin{align}
\beta(\cE_{1},\cE_{2})=\lim_{k}\lim_{l}\beta(\cE_{1,l},\cE_{2,k}),
\end{align}
and using elementary considerations on $\beta(\cE_{1},\cE_{2})$ when $\cE_{j}$ are finite fields. For a proof under slightly more restrictive hypotheses, see \cite[Proposition F.2.8]{doumarcha}.}. This representation holds in particular if $\cE_{k}:=\sigma(Z_{k})$ is the sigma algebra generated by $Z_{k}$, where $Z_{k}$ ($k=1,2$) is a random element of a Polish space $S_{k}$.}
 
  Additionally, it follows  that 
   \begin{enumerate}[(i)]
   \item \label{itemi}{\it The $\beta-$mixing coefficients are symmetric: $\beta(\cE_{1},\cE_{2})=\beta(\cE_{2},\cE_{1})$}. 
   \item \label{itemiii}{\it $\beta(\cdot,\cdot)$ is increasing in each component: if $\cE_{k}'\subset \cE_{k}$ ($k=1,2$)    then}
\begin{equation}
\label{betmixareinc}
\beta(\cE_{1}',\cE_{2}')\leq \beta(\cE_{1},\cE_{2}).
\end{equation}
   \item \label{itemii}
   {\it $\beta(\cE_{1},\cE_{2})=0$ if and only in $\cE_{1}$ and $\cE_{2}$ are $\P-$independent}.      
\end{enumerate}
The first two properties  follow by the equality between the extreme sides of \eqref{equcharbetmixcoefuse}. The third one is clear even from the general definition \eqref{betmixdef}. 
\end{remark}

\subsubsection{$\beta-$coefficients of $m-$dependence}
\nch{We now extend the previous considerations to a case involving families of sub-sigma algebras related to a sequence of random variables $Z_{1:\infty}$. The aim is to set a precise discussion involving some $\beta-$mixing coefficients associated to ``the  present'' and ``the past'' of this sequence.
}
\begin{dfn}[ $\beta-$coefficients of $m-$dependence]
\label{defbetmixpaspre}
 Given a  subset $J\subset \N$, a random element $Z_{J}$ of $S^\otimes_{J}$, and $(m,l)\in \N\times \N$,  the $l-$th $\beta-$coefficient of $m-$dependence  of 
$Z_{J}$ is defined as
\begin{align}
\label{betdepcoeequ}
\beta_{Z_J}(m,l):=\beta(\sigma(Z_{J\cap[1,l-m]}),\sigma(Z_{J\cap \{l\}})),
\end{align}
where the right--hand side is defined in \eqref{betmixdef} and with the convention $Z_{\emptyset}:=\emptyset$. The  maximal $\beta-$coefficient of $m-$dependence is denoted by $\beta_{Z_{J}}(m)$:
\begin{align}
\label{equsupbetdepcoe}
\beta_{Z_{J}}(m):=\sup_{l\in \N}\beta_{Z_{J}}(m,l).
\end{align} 
\end{dfn}
Thus for $l\in J$, $\beta_{Z_{J}}(m,l)$ gives the $\beta-$mixing coefficient between $Z_{l}$ and the ``distant past'' (at least $m$ units before $l$) of $Z_{J}$. Similarly,  $\beta_{Z_{J}}(m)$ is the \nch{smallest upper bound of the} $\beta-$mixing coefficients of $Z_{J}$ within   ``some present'' and its (at least $m$ units) ``distant past''.

We list, for future reference, some  properties  of $\beta_{Z_{\cdot}}(\cdot,\cdot)$ and $\beta_{Z_{\cdot}}(\cdot)$.
\begin{properties}[of $\beta_{Z_{J}}$]
\label{prop:remrbetzjinwor}~
\begin{enumerate}
\item \label{betzerknotinJ} If $l\notin J$ then $\beta_{Z_{J}}(m,l)=0$ for all $m$, see \nch{\eqref{itemii} in} Remark \ref{remprobetmix}.
\item \label{betzjaredec} For fixed $l$, $\beta_{Z_{J}}(\cdot, l)$ is decreasing. Thus \nch{$(\beta_{Z_{J}}(m))_{m}$} is also decreasing.
\item \label{itemdepsuf} A sufficient condition for ${\beta_{Z_{J}}}(m)=0$ is the {\it $m-$dependence of  $Z_{J}$}, i.e., the  hypothesis that for every $l$, $Z_{J\cap[1,l]}$ and $Z_{J\cap[l+m,\infty)}$  are independent (this condition is not necessary\footnote{\label{foonotbetnotbetmix} Choose random variables $X,Y,Z$ with $X$ independent of $Y$ and  $X$ independent of $Z$ but with $Y+Z$ not independent of $X$, choose $X'$ independent of $\sigma(X,Y,Z)$ and consider, for $n=4$, $J=\{1,2,3,4\}$, and $m=2$, the choices $Z_{1}=X$, $Z_{2}=X'$, $Z_{3}=Y$, $Z_{4}=Z$).}). In particular, $\beta_{Z_{J}}(\cdot)\equiv 0$ if the entries  of $Z_{J}$ are independent.
\item If $J'\subset J$, then ${\beta_{Z_{J'}}}(\cdot,\cdot)\leq {\beta_{Z_{J}}}(\cdot,\cdot)$
\nch{(pointwise)} by \eqref{betmixareinc}. If in particular $Z_{1:n}$ is a random element of $S^\otimes_{1:n}$ and $J\subset \{1,\dots,n\}$ is any subset then
\begin{align}
\label{equcomzjj1ton}
{\beta_{Z_{J}}}(\cdot,\cdot)\leq {\beta_{Z_{1:n}}}(\cdot,\cdot).
\end{align}

\item Assume that the partition $\cJ$ \nch{of $\{1,\dots,n\}$} is such that the indexes within each $J\in \cJ$ are separated by a ``minimal gap'', say $1\leq m< n$,\footnote{I.e., $m$ is the smallest $m'$ such that, for any $J\in \cJ$ and any different $j_{1},j_{2}\in J$, $|j_{1}-j_{2}|\geq m'$.} then the inequality
\begin{align}
\label{equbetmdis}
\beta_{Z_{J}}(m')\leq 
\beta_{Z_{J}}(1)\leq  \beta_{Z_{1:n}}(m)
\end{align}
holds for all $m'\in \N$ and all $J\in \cJ$. The first inequality follows from Property \ref{betzjaredec}, the second follows from  Property \ref{betzerknotinJ}, 
the inequality \eqref{equcomzjj1ton}, and the fact that for all $J\in \cJ$ and all $l\in J$, $J\cap[1,l-1] \subset \{1,\dots, l-m\}$. 
 
\end{enumerate} 
\end{properties}

\subsubsection{Examples}
 Of particular interest for us are the following mixing hypotheses on the rate of decay of $\beta_{Z_{\cdot}}(\cdot)$.
\begin{dfn}[Sub exponentially $\beta-$mixing process]
\label{defsubbetmix}
Let  $Z_{1:\infty}$ be a random element of $S^{\otimes}_{1:\infty}$. $Z_{1:\infty}$ is {\it subexponentially $\beta-$mixing} with parameters $(a,b,\gamma)\in (0,\infty)\times(0,\infty)\times(0,\infty)$ if for all $m\in \N$
\begin{align}
\label{equdefexpmix}
\beta_{Z_{1:\infty}}(m)\leq a\exp(-bm^{\gamma}).
\end{align}
\end{dfn}
\begin{dfn}[Subpolinomially $\beta-$mixing processes]
\label{defsubpolmix}
Let  $Z_{1:\infty}$ be a random element of $S^{\otimes}_{1:\infty}$. $Z_{1:\infty}$ is {\it subpolinomially $\beta-$mixing} with parameters $(a,\gamma)\in (0,\infty)\times(1,\infty)$ if for all $m\in \N$
\begin{align}
\label{equdefpolmix}
\beta_{Z_{1:\infty}}(m)\leq am^{-\gamma}.
\end{align}
\end{dfn} 
 If $Z_{1:\infty}$  is a Markov Chain with state space $S$, then, by the Markov property,
\begin{align}
\beta_{Z_{1:\infty}}(m)=\sup_{n}\beta(\sigma(Z_{n}),\sigma(Z_{n+m})).
\end{align} 
Sufficient conditions for exponentially mixing rates (i.e. \eqref{equdefexpmix} holds with $\gamma\geq 1$) of Markov chains can be consulted also in \cite{brad:05}, \nch{\cite{fort:moul:03}, \cite[Chapter 16]{meyn:twee:09}}. For sufficient conditions implying (in the Markovian setting) subexponential $\beta-$mixing rates  \eqref{equdefexpmix} with $\gamma\in (0,1)$ or polynomial rates like \eqref{equdefpolmix}, see for instance \cite{tuotwe}, \cite{jarrob}, \nch{\cite{fort:moul:03b}, \cite{douc:fort:moul:soul:04},} \cite{douforgui},  and the references therein.

\subsection {Berbee's Lemma}
The $\beta-$mixing coefficients measure, on a certain sense, the ``($\P-$)distance from independence'' between two sigma--algebras. This notion is put forward in a more concrete way by the following classical coupling result\footnote{We omit specifications about the ``richness'' of $(\Omega,\cA)$, which are implicitly embedded in the introductory remarks.} (see \cite[Corollary 4.2.5]{ber}, \cite[Theorem 1, p.7]{douk:12}).
\begin{lemma}[Berbee's Lemma]
\label{berlem}
Let {$(V,W)$} be a random vector in $S_{1}\times S_{2}$. There exists a {$S_{2}-$}valued random vector {$W^{*}$}, distributed as {$W$},  independent of $V$, and with the property 
\begin{equation}
\label{berlemequ}
\beta(\sigma(V),\sigma({W}))=\P({W\neq W^{*}}).
\end{equation}
\end{lemma}

The above lemma admits the following (apparently) generalised version\footnote{Whose proof, although developed independently, follows an argument resembling the one in \cite[p.484]{vien:97}.}:

\begin{lemma}[Generalised Berbee's Lemma]
\label{genberlem}
Given $N\in \N$ and a random sequence ${V}_{1:N}$ of $S^\otimes_{1:N}$, there exists a random sequence ${V}_{1:N}^{*}$ with independent entries  such that for every $1\leq k\leq N$,
\begin{enumerate}
\item ${V}_{k}$ and ${V}_{k}^{*}$ have the same distribution, and
\item 
\begin{align}
\label{equberlemranvardif}
\P({V}_{k}\neq {V}_{k}^{*})=\beta(\sigma({V}_{1:k-1}),\sigma({V}_{k})).
\end{align}
\end{enumerate}
(In particular, ${V}_{1}={V}_{1}^{*}$, $\P-$a.s.)
\end{lemma}
\begin{proof} 
\nch{We start with a preliminary observation: notice that if ${V}_{1}, {V}_{2}, V$ are  random variables with $V$ independent of $\sigma({V}_{1},{V}_{2})$ then for any Borel set $E_{2}\subset S_{2}$, 
\begin{align}
\Prob{{V}_{2}\in E_{2}|\sigma({V}_{1})}= \Prob{{V}_{2}\in E_{2}|\sigma({V}_{1},{V})},
\end{align} 
$\P-$a.s. Using the characterization in the first equality of \eqref{equcharbetmixcoefuse} (with obvious adjustments on notation) and the symmetry of $\beta(\cdot,\cdot)$,
\begin{align}
\beta(\sigma({V}_{1},{V}),\sigma({V}_{2}))=&\Esp{\sup_{n}|\Prob{V_{2}\in E_{2,n}}-\Prob{V_{2}\in E_{2,n}|\sigma(V_{1},V)}|} & \\
=& \Esp{\sup_{n}|\Prob{V_{2}\in E_{2,n}}-\Prob{V_{2}\in E_{2,n}|\sigma(V_{1})}|}&=\beta(\sigma({V}_{1}),\sigma({V}_{2})).
\label{equequbetmixind}
\end{align}}
Now we prove the statement. First, we assume that $N\geq 2$ (otherwise the conclusion is trivial, even without the vacuous property of independence, for $V_{1}^{*}:=V_{1}$).

Let now ${V}_{1:N}$ be a  random sequence in $S^\otimes_{1:N}$. We will construct a sequence ${V}_{2:N}^{*}$ satisfying, for all $1\leq k <N$, the  property {\bf P(k)} defined by 

\medskip

\begin{description}\item {\bf P(k):} The sequence $V_{k+1:N}^{*}$  is such  that, for $k\leq j< N$, 
\begin{enumerate}
\item $V_{j+1}$ and $V_{j+1}^{*}$ are identically distributed with $\beta(\sigma(V_{1:j}),\sigma(V_{j+1}))=\P(V_{j+1}\neq V_{j+1}^{*})$.

\item The vectors  $V_{1:j}$ and $V_{j+1:\nch{N}}^{*}$ are independent.
\end{enumerate}
\end{description}
which is easily seen to be sufficient to prove the claim of Lemma \ref{genberlem} by defining $V_{1}^{*}:=V_{1}$.

We will construct $V_{2:N}^{*}$ by backward induction: start defining $V_{N}^{*}$  by applying Lemma \ref{berlem} with ${V}={V}_{1:N-1}$ and ${W}={V}_{N}$. This verifies the satisfaction of {\bf P(N-1)}.

Now, assume that {\bf P(k)} has been verified by $V_{k+1:N}^{*}$ for some $1\leq k<N$. An application of Berbee's lemma with $V:=(V_{1:k-1},V_{k+1:N}^{*})$  and  $W:=V_{k}$ guarantees the existence of a random variable $V_{k}^{*}$ distributed as $V_{k}$ and independent of $\sigma(V_{1:k-1},V_{k+1:N}^{*})$ such that  
\begin{align}
\label{equproberlemgen}
\P(V_{k}\neq V_{k}^{*})=\beta(\sigma(V_{1:k-1},V_{k+1:N}^{*}),\sigma(V_{k}))=
\beta(\sigma(V_{1:k-1}),\sigma(V_{k})),
\end{align} 
where the last equality follows by an application of \eqref{equequbetmixind}. The augmented sequence $V_{k:N}^{*}$ satisfies therefore {\bf P(k-1).} After $N-1$ steps this gives the desired construction. 
\end{proof}
 
\subsectionmark{A general estimate}
\subsection{A general estimate for decoupled averages}
\subsectionmark{A general estimate}

Our next result, Theorem \ref{theconbet}, is an  inequality relating the distribution function of certain random variables defined by suprema and associated  to a {composable pair $(Z_{1:n}, \cG_{1:n})$} to the corresponding distribution functions over sets of indexes in a partition $\cJ$ of $\{1,\dots,n\}$ and a ``decoupling'' of $Z_{1:n}$ over each one of the set of indexes in $\cJ$.

Indeed, the inequality \eqref{equbousumpar}, combined with Lemma \ref{genberlem}, allows us to relate distribution functions as in the left--hand side of \eqref{equbousumpar} to a sum of similar distribution functions {\it which are defined for independent sequences}, controlling the additional error with the $\beta-$dependence coefficients associated to $\cJ$ in Definition \ref{defbetmixpaspre}.

\begin{thm}
\label{theconbet}
Let $\cJ$ and  $(Z_{1:n},\cG_{1:n})$ be as in Lemma \ref{thedisfungen}. There exists a sequence $Z_{1:n}^{*}$ with the following properties:
\begin{enumerate}
\item \label{prosammar} For every $k\in \{1,\dots,n\}$, the distributions of $Z_{k}^{*}$ and $Z_{k}$ are the same. 
\item \label{proindblo} For every $J\in \cJ$, $Z_{J}^{*}$ is an independent sequence.

\item \label{procomine} The inequality
 \begin{align}
 \begin{split}
&\Prob{\sup_{{g_{1:n}}\in\cG_{1:n}}(a\emzon+b\mmzon) g_{1:n} \geq t}\\
&\qquad \leq \sum_{J\in \cJ}\left(\Prob{\sup_{{ g_{J}}\in\cG_{J}}(a\emzJsta+b\mmzJ) g_{J}\geq t}+\sum_{k\in J}\beta_{Z_{J}}(1,k)\right).
 \end{split}
 \label{equbousumparind}
\end{align} 
holds \nch{for every $(a,b,t)\in \R^{3}$.}
\end{enumerate}
\end{thm}

\begin{proof}
We start by an application of Lemma \ref{thedisfungen}, the next step is a further estimate of the right--hand side of \eqref{equbousumpar} via Lemma \ref{genberlem}.

Indeed, fix $J\in \cJ$.  We apply Lemma \ref{genberlem} to $Z_{J}$  to construct $Z_{J}^{*}$: if $J:=\{j_{1},\dots ,j_{N}\}$ in increasing order, construct $Z_{J}^{*}$ by replacing $V_{k}:= Z_{j_{k}}$ in Lemma \ref{genberlem}. 

Properties \ref{prosammar}. and \ref{proindblo}. are immediate from this construction. Notice also that, by the construction and \eqref{equberlemranvardif}, 
\begin{align}
\label{equappberpro}
\P(Z_{k}\neq Z_{k}^{*})=\beta_{Z_{J}}(1,k), \qquad \forall k \in J.
\end{align}

Now, using the inclusion 
\begin{align}
\label{equestberlemequ}
\bigg\{\sup_{{ g_{J}}\in\cG_{J}}(a\emzJ+{b}\mmzJ) g_{J}\geq t\bigg\}
\subset \bigg\{\sup_{{ g_{J}}\in\cG_{J}}(a\emzJsta+\nch{b}\mmzJ)g_{J}\geq t\bigg\}\cup \bigcup_{k\in J}\{Z_{k}\neq Z_{k}^{*}\},
\end{align} 
  \eqref{equbousumparind} follows from the union bound via  Lemma \ref{thedisfungen} and \eqref{equappberpro}.
\end{proof}

\nch{\begin{remark}[Complement to Remark \ref{remgenestfamfun}]\label{remvergenfamfunlif} Let $\{K_{J}\}_{J}$ be a family of functionals as in Remark \ref{remgenestfamfun}, then exactly the same argument as in the proof of Theorem \ref{theconbet} gives that, for every finite $I\subset \N$ and every partition $\cJ$ of $I$ 
\begin{align}
 \Prob{K_{I}(Z_{I}) \geq t}
 \leq \sum_{J\in \cJ}\left(\Prob{K_{J}(Z_{J}^{*})\geq t}+\sum_{k\in J}\beta_{Z_{J}}(1,k)\right),
 \label{equbousumparindgen}
\end{align}
where $Z_{I}^{*}$ satisfies properties \ref{prosammar}. and \ref{proindblo}. above (with $\{1,\dots,n\}$ replaced by $I$). 
\end{remark}

\begin{remark}[Relationship with Bernstein's method]
Let $\{K_{J}\}_{J}$ be again as in Remark \ref{remgenestfamfun}. Given a finite set $J\subset \N$, a random element $Z_{J}$ of $S_{J}^{\otimes}$, a partition $\{J_{1},\dots,J_{r}\}$ of $J$, and a  partition $I_{1},\dots, I_{s}$ of $\{1,\dots, r\}$,   denote, for every $k\in \{1,\dots,r\}$, $J(I_{k}):=\cup_{s\in I_{k}}J_{s}$. Then using an argument similar to the one in the proof of Theorem \ref{theconbet} it is easy to prove that 
\begin{align}
\label{equestlifbermet}
\Prob{K_{J}(Z_{J})\geq t}\leq \sum_{k=1}^{r}\left(\Prob{K_{J({I_{k}})}(Z_{J({I_{k}})}^{**})\geq t}+\sum_{j\in I_{k}}\beta_{(Z_{J_{l}})_{l\in I_{k}}}(1,j)\right)
\end{align}
where 
\begin{enumerate}
\item $Z_{J_{k}}$ and $Z_{J_{k}}^{**}$ have the same distribution, for $k\in 1,\dots, r$.
\item For every fixed $k \in \{1,\dots, r\}$, the sequence $(Z_{J_{l}}^{**})_{l\in I_{k}}$ is independent.
\end{enumerate}  
Assume that $n=2am$ for some $(a,m)\in\N\times\N$. Then one can group the sequence $Z_{1:n}$ into $2m$ disjoint  blocks $J_{1},\dots,J_{2m}$ of successive elements, each of length $a$, and classify the blocks in ``odd'' and ``even'' blocks, which corresponds in \eqref{equestlifbermet} to taking $I_{1}$ and $I_{2}$ as (respectively) the odd and even numbers in $\{1,\dots,2m\}$. An application of \eqref{equestlifbermet} (for $J=\{1,\dots, n\}$) together with an easy adaptation of \eqref{equbetmdis} gives the estimate (with a slight abuse of notation)
\begin{align}
\Prob{K(Z_{1:n})\geq t}\leq \sum_{k=1}^{2}\Prob{K((Z_{J_{l}}^{**})_{l \in I_{k}})\geq t}+2m\beta_{Z_{1:n}}(a),
\end{align}
where, for fixed $k$, each sequence $(Z_{J_{l}}^{**})_{l\in I_{k}}$ is independent. This is the key idea in ``Bernstein's partition method'' (see for instance \cite{KM16} and the references therein). 

The most important   difference between the partitions used in the estimates \eqref{equbousumparindgen} and \eqref{equestlifbermet} is that, in \eqref{equbousumparindgen}, there is dependence {\it within} the blocks $(Z_{J}^{*})_{J\in \cJ}$ but there is independence {\it inside} each block $Z_{{J}}^{*}$. In \eqref{equestlifbermet} the situation is somewhat reversed: {\it for fixed} $k\in \{1,\dots, r\}$, there is independence {\it within} the blocks $(Z_{J_{l}}^{**})_{l\in I_{k}}$, but there is dependence {\it inside} each block $Z_{J_{l}}^{**}$. The reader is invited to consider the consequences of this difference for what follows. 
\end{remark}}

\subsection{Abstract lifting of deviation inequalities}
\label{secabslifdevine}
In this {concluding} part we present a result indicating how to ``lift''  deviation inequalities from the independent to the (possibly) dependent case via Theorem \ref{theconbet}.  The purpose of this result for what follows is to serve as an intermediate step towards the beta--mixing generalization of the deviation estimates proved in  \cite{bargob}. The notation and conventions are those explained in Section \ref{secnot}.

\begin{prop}[Abstract lifting of deviation inqualities]
\label{prop_corsemabsexp} \nch{Let $n\in \N$,}
 let $(a,b,B)\in \R\times \R\times (0,\infty]$, and let $(Z_{1:n},\cG_{1:n})$ be a composable pair such that 
\begin{align}
\label{equconcalgbou}
\sup_{ { g_{1:n}}\in \cG_{1:n}}\sup_{1\leq k\leq n}
||g_{k}(Z_{k})||_{\P,\infty}\leq B.
\end{align}
Moreover,  assume that  there exists a function 
 \begin{align}
 \label{equlabdef}
 L_{a,b}:\{1,\dots, n\}\times [0,\infty]\to [0,\infty) 
 \end{align}
such that for any   $t\geq 0$,  $J\subset \{1,\dots,n\}$  
and some  $Z_{J}^{*}$ with independent entries and the same marginals as $Z_{J}$, we have
\begin{align}
\label{equhypestindsammar}
\Prob{\sup_{ g_{J}\in\cG_{J}}(a\emzJsta+b\mmzJ) g_{J}\geq t}\leq L_{a,b}(|J|,t).
\end{align}
\nch{Let $m\in \{1,\dots,n\}$ and write
\begin{align}
\label{eucalg}
n:=qm+r, \quad \text{with}\ q=\left\lfloor\frac{n}{m}\right\rfloor, \quad  0\leq r <m
\end{align}
for the Euclidean algorithm for $n$ divided by $m$}; then the estimate
\begin{align}
&\Prob{\sup_{ g_{1:n}\in\cG_{1:n}}(a\emzon+b\mmzon) g_{1:n}\geq t} \\[3mm]
 &\qquad \leq (rL_{a,b}(q+1,t)+(m-r)L_{a,b}(q,t)+n\beta_{Z_{1:n}}(m))
\1_{\{t\leq (|a|+|b|)B\}}\\[3mm]
&\qquad\leq (m(L_{a,b}(q+1,t)\lor L_{a,b}(q,t))+n\beta_{Z_{1:n}}(m))
\1_{\{t\leq (|a|+|b|)B\}}
\label{equsemabsconine}
\end{align}
holds (with the convention $L_{a,b}(n+1,t)\equiv L_{a,b}(n,t)$).
\end{prop}

\nch{The  inequality \eqref{equhypestindsammar} as an assumption is quite standard: it says that for a given class of functions $\cG_{J}$, the uniform deviation depends of the size of the sample $|J|$ and the amplitude of the deviation $t$, consistently with many results on uniform deviation inequalities (see for instance \cite{gyor:kohl:krzy:walk:02, ledo:tala:13}). As we will see later, the  complexity of the class $\cG_{J}$ typically appears in $L_{a,b}(.)$.}
\begin{proof}
The second inequality in \eqref{equsemabsconine} is trivial
. We proceed to prove the first inequality. In view of the Euclidean decomposition 
$n=qm+r$,
consider the {\it $m-$steps partition} 
\begin{align}
\label{equdefjmste}
\cJ_{m\, steps}:=\{J_{1},\dots, J_{m}\}
\end{align}
  of $\{1,\dots, n\}$ specified by 
\begin{align}
J_{k}:=\big\{k+lm\big\}_{l=0}^{q}, & \qquad 1\leq k\leq r \\
J_{k}:=\big\{k+lm\big\}_{l=0}^{q-1}, & \qquad r<k\leq m.
\label{mpluonesteparequ}
\end{align}
In words, $J_{k}$ is the set obtained by starting from $k$ and moving to the right in steps of $m$ units as far  as possible before quitting the set $\{1,\dots,n\}$. Clearly, $|J_{k}|=q+1$ for $1\leq k\leq r$ and $|J_{k}|=q$ for $r<k\leq m$.

We apply Theorem \ref{theconbet} with $\cJ:=\cJ_{m\,steps}$. This gives the  upper bound
\begin{align}
\left(rL_{a,b}(q+1,t)+(m-r)L_{a,b}(q,t)+\sum_{k=1}^{m}\sum_{j\in J_{k}}{{\beta}_{Z_{J_{k}}}(1,j)}\right)\1_{\{t\leq (|a|+|b|)B\}}
\label{estthebetmixineori}
\end{align}
for the left--hand side of \eqref{equsemabsconine}. \nch{The conclusion follows using  the estimate \eqref{equbetmdis} which gives $\sup_{j\in J_k}\beta_{Z_{J_k}}(1,j)\leq \beta_{Z_{J_k}}(1) \leq \beta_{Z_{1:n}}(m)$
and the fact that  $\sum_{k=1}^{m}\sum_{j\in J_{k}}1=n$. 
}
\end{proof}

\section{Some applications to nonparametric regression}
\label{secapp}
\sectionmark{Applications to nonparametric regression}
In this section, we develop some of the  applications of the results in Section \ref{secbribetmixind} to the problems addressed, in the context of independent samples, within \cite{bargob}. The notation, again, comes from Section \ref{secnot}.

\subsection{Empirical covering numbers}

\subsectionmark{Covering numbers}
\nch{The functions $L_{a,b}$ in \eqref{equhypestindsammar} usually depend on the complexity of the functions class $\cG_J$, through  its covering number w.r.t.  a suitable semimetric, see the seminal work \cite{VC71}. We now recall the notion of $r-${\it coverings} and {\it covering numbers}, taking care of extending it to our case of sequences of spaces $\cG_J$.}

\begin{dfn}[$r-$covering, covering numbers]
\label{defrcov}
\nch{Let $(\cG,d)$ be a semimetric space}, let $\cG_{0}\subset \cG$, and let $r\in [0,\infty)$. An $r-$covering of $\cG_{0}$ with respect to $d$ is a set $\cG'\subset \cG$ with the property that, for every $g\in \cG_{0}$, there exists $g'\in \cG'$ satisfying
\begin{align}
\label{equdefcon}
d(g,g')<r.
\end{align}
The $r-$covering number of $\cG_{0}$ with respect to $d$ is defined as
\begin{align}
\label{equdefrcovnum}
\cN^{(d)}(r,\cG_{0}):=\min\{|\cG'|: \mbox{\, $\cG'\subset \cG$ is an $r-$covering of $\cG_{0}$ with respect to $d$}\}.
\end{align}
\end{dfn}

Notice that the meaning of $\cN^{(d)}(r,\cG_{0})$ depends not only on the set $\cG_{0}$ and the metric $d|_{_{\cG_{0}\times\cG_{0}}}$, but also on the space $\cG$ where $d$ is defined.

The following type of covering numbers are of special relevance for us.
\begin{dfn}[Empirical covering numbers]
\label{defempcovnum}
Let $J\subset \N$ be a finite set,  let $\cG_{J}\subset \cL^{\otimes_{J}}_{S}$ be a sequential family of functions, and let $z_{_{J}}\in S_{J}^{\otimes}$ be given. We define the  empirical $L_{1}$ $r-$covering numbers of $\cG_{J}$ at $z_{_{J}}$, $\cN_{1}(r,\cG_{J},z_{_{J}})$, as
\begin{align}
\cN_{1}(r,\cG_{J},z_{_J}):=\cN^{(d^{1}_{z_{_J}})}(r,\cG_{J}),
\end{align}
where $d_{z_{_J}}^{1}$ is the empirical $L^{1}-$seminorm $d^{1}_{z_{_J}}(g_{J},g_{J}'):=\emzsmaJ |g_{J}-g_{J}'|$ on the product space  $\cL^{\otimes_{J}}_{S}$.
\end{dfn}

\begin{remark}[Measurability issues]
\label{remmeacovnum}
It is clear that, if $Z_{J}$ is a random element of $S^{\otimes}_{J}$, $\omega\mapsto \cN_{1}(r,\cG_{J},Z_{J}(\omega))$ is a nonnegative function. To avoid unnecessary measurability discussions, we will denote by $\Esp{\cN_{1}(r,\cG_{J},Z_{J})}$ the {\it outer} expectation of $\cN_{1}(r,\cG_{J},Z_{J})$: 
\begin{align}
\label{outexpcovnum}
\Esp{\cN_{1}(r,\cG_{J},Z_{J})}:=\inf_{h}\Esp{h},
\end{align}
where the infimum is taken over the random variables $h:\Omega\to \R$ with $\cN_{1}(r,\cG_{J},Z_{J})\leq h$ (except on a set of $\P-$measure zero), with the convention $\inf\emptyset=\infty$. 
\end{remark}

\subsection{Uniform deviation inequalities for  dependent samples}

\subsectionmark{Uniform deviations}
We start by recalling  the following result, which is a  consequence of \cite[Theorem 2.2]{bargob} (with easy simplifications left to the reader). \nch{We will use it as a ``toy'' theorem, whose extension to the dependent case will  illustrate some  arguments that are not written in detail later.}
\begin{thm}[\nch{Uniform deviation probability, independent version}]
\label{the22bargob} Let $X_{1:n}$ be a random element of $(\R^{d})^{n}$ with independent entries, 
and assume that $(X_{1:n},\cF_{1:n})$ is a composable pair where $\cF_{1:n}$ is a pointwise measurable sequential family\footnote{I.e. such that there exists $\{f_{1:n}^{(k)}\}_{k}\subset \cF_{1:n}$ with the property that, for every $f_{1:n} \in \cF_{1:n}$, there exists a sequence $(k_{l})_{l}$ satisfying $\lim_{l} f_{1:n}^{(k_{l})}=f_{1:n}$ pointwise.} 
with $f_{k}:\R^{d}\to [0,B]$ ($k=1,\dots,n$) for some $B>0$  and for each $f_{1:n}\in \cF_{1:n}$. Then for
\begin{align}
\label{equparthe22bargob:1}
\nch{
(\epsilon,c,\gamma,\gamma')\in 
\times(0,1)\times(1,\infty)\times(1,\infty)\times(1,\infty)},
\end{align}
the estimate
\begin{align}
&\Prob{\sup_{f_{1:n}\in \cF_{1:n}}((1-\epsilon)\emxon-(1+\epsilon)\mmxon)f_{1:n}>t}\\
&\quad \leq
\frac{2\gamma}{\gamma-1}\Esp{\cN_{1}(\frac{1}{2}u_{1}(c,\gamma')\,t,\cF_{1:n},X_{1:n})}\exp(-\frac{1}{2B}u_{2}(c,\gamma')\epsilon nt)
\label{equthe22bargob}
\end{align}
holds with \begin{align}
\label{equdefuj}
u_{1}(c,\gamma'):=(1-\frac{1}{c})\frac{1}{\gamma'}, &\qquad u_{2}(c,\gamma'):= (1-\frac{1}{c})^{2}(1-\frac{1}{\gamma'}),
\end{align} 
provided that 
\begin{align}
\label{equconthe22bargob}
t\geq \frac{Bc}{2}\left(\frac{\gamma }{n}\right)^{1/2}.
\end{align}
\end{thm}

\nch{Our extension of this result will be made with the help of the following notion, which we will discuss briefly in Section \ref{remcovnumothmet}:}
\begin{dfn}[Uniform $L^{1}-$entropy estimates]
\label{defl1entest}
Let $J\subset \N$ and let $\cG_{J}\subset \cL^{\otimes_{J}}_{S}$ be given. A Borel-measurable function $\lambda: \N \times(0,\infty)\to [1,\infty]$ is called an {empirical $L_{1}-$\it uniform entropy estimate of $\cG_{J}$} (or simply, a uniform entropy estimate of $\cG_{J}$) if for every finite subset $J'\subset J$
and every $r\in(0,\infty)$
\begin{align}
\label{equdefl1entest}
\log\left(\sup_{z_{J'}\in S^{\otimes}_{J'}}\cN_{1}(r,\cG_{J'},z_{_{J'}})\right)\leq\lambda(|J'|,r). 
\end{align}
\end{dfn}

\nch{Going back to the extension of Theorem \ref{the22bargob}, } 
\nch{assume the existence of a uniform entropy estimate $\lambda$ of $\cG_{1:n}$, then $\lambda$ is clearly a uniform entropy estimate of $\cG_{J}$ for every $J\subset \{1,\dots,n\}$; in addition, for any random element $Z_{J}$  of $S^{\otimes}_{J}$, we have
\begin{align}
\label{estespcovlam}
\Esp{\cN_{1}(r,\cG_{J},Z_{J}))}\leq \exp(\lambda(|J|,r)).
\end{align}
Consequently,} under the hypotheses of Theorem \ref{the22bargob}, we have \nch{that} the inequality
\begin{align}
&\Prob{\sup_{f_{J}\in \cF_{J}}((1-\epsilon)\emxJ-(1+\epsilon)\mmxJ)f_{J}>t}  \\
&\quad\leq
\frac{2\gamma}{\gamma-1}\exp\left(-\frac{1}{2B}u_{2}(c,\gamma')\epsilon |J|t+\lambda\Big(|J|,\frac{1}{2}u_{1}(c,\gamma')\,t\Big)\right)
=: L_{c,\gamma,\gamma',\epsilon}(|J|,t)\qquad
\label{equthe22bargobent}
\end{align}
holds for every $J\subset \{1,\dots, n\}$, provided this time that 
$t\geq \frac{Bc}{2}(\frac{\gamma}{|J|})^{1/2}.$
We can ``hide'' this restriction on $t$ by extending \eqref{equthe22bargobent} to the estimate
\begin{align}
\Prob{\sup_{f_{J}\in \cF_{J}}((1-\epsilon)\emxJ-(1+\epsilon)\mmxJ)f_{J}>t}
\leq \1_{\{t< \frac{Bc}{2}\left(\frac{\gamma }{|J|}\right)^{1/2}\}} + L_{c,\gamma,\gamma',\epsilon}(|J|,t)\1_{\{t\geq \frac{Bc}{2}\left(\frac{\gamma }{|J|}\right)^{1/2}\}},
\label{equthe22bargobenthid}
\end{align}
which holds for every $J\subset \{1,\dots,n\}$ and every $t>0$, always under the hypotheses of Theorem \ref{the22bargob}.  This, together with Proposition \ref{prop_corsemabsexp}, allows us to deduce the following ``$\beta-$ version'' of Theorem \ref{the22bargob}.
\begin{thm}[\nch{Uniform deviation probability, $\beta-$version}]
\label{the22bargobbetmix}
Let $X_{1:\infty}$ be a random sequence in $(\R^{d})^{\N}$. For $n\in \N$,  assume that $(X_{1:n},\cF_{1:n})$ is a composable pair where $\cF_{1:n}$ is a pointwise measurable sequential family and each $f_{1:n}\in \cF_{1:n}$ is a sequence of functions  with $f_{k}:\R^{d}\to [0,B]$ ($k=1,\dots,n$) for some $B>0$. Assume that $\lambda$ is a uniform entropy estimate of $\cF_{1:n}$ 
(Definition \ref{defl1entest}),  and let
\begin{align}
\label{equparthe22bargob:2}
\nch{(
\epsilon,c,\gamma,\gamma')\in  
(0,1)\times(1,\infty)\times(1,\infty)\times(1,\infty).}
\end{align}
Then, with $u_{j}$ ($j=1,2$) as in \eqref{equdefuj}, with $L_{c,\gamma,\gamma',\epsilon}:\{1,\dots,n\}\times [0,\infty)\to [0,\infty)$ as in \eqref{equthe22bargobent}  and with $\beta_{X_{1:\infty}}(\cdot)$ as in \eqref{equsupbetdepcoe}, the estimate
\begin{align}
&\Prob{\sup_{f_{1:n}\in \cF_{1:n}}((1-\epsilon)\emxon-(1+\epsilon)\mmxon)f_{1:n}\geq t}\\
&\quad\leq 
\bigg[m\Big(L_{c,\gamma,\gamma',\epsilon}(\left\lfloor\frac{n}{m}\right\rfloor,t)\lor L_{c,\gamma,\gamma',\epsilon}(\left\lfloor\frac{n}{m}\right\rfloor+1,t)\Big)+n\beta_{X_{1:\infty}}(m)\bigg]\1_{\{ t\leq 2B\}},
\end{align}
holds for every $m\in \{1,\dots,n\}$ (with the convention $L_{c,\gamma,\gamma',\epsilon}(n+1,t)\equiv L_{c,\gamma,\gamma',\epsilon}(n,t)$), provided that
\begin{align}
t\geq \frac{Bc}{2}\left(\frac{\gamma }{\left\lfloor\frac{n}{m}\right\rfloor}\right)^{1/2}.
\end{align}
\end{thm}

\nch{In practice, the choice of $m$ will depend on the applications at hand. Typically, it will be done with the goal of minimizing in a convenient way the deviation from the rates obtained in the independent case (towards optimal extensions of the results in \cite{bargob}) that follow from the results under consideration (see for instance Propositions \ref{corwearatsubexpmix} and \ref{corwearatsubpolmix} below
).
}  
The uniform deviations proved in \cite{bargob} for independent samples (\cite[Section 2]{bargob}) can be extended in a similar manner.

\subsection{Remarks on entropy estimates}

\label{remcovnumothmet}
The definition of uniform entropy estimates, given here as a uniform estimate of the covering numbers associated to the $L^{1}-$ empirical seminorm  (Definition \ref{defempcovnum}),  can of course be extended via Definition \ref{defrcov} to other families of semimetrics in $\cL^{\otimes_{J}}_{S}$, such as the empirical $L^{p}-$seminorms ($p\geq 1$) defined as $d^{p}_{z_{J}}:=(A_{z_{J}}|f_{J}-g_{J}|^{p})^{1/p}$. The relationships between  covering numbers for different semimetrics can be relevant:  for instance, it is clear from the Cauchy-Schwarz inequality that $d^{1}_{z_{J}}(\cdot,\cdot)\leq d^{2}_{z_{J}}(\cdot,\cdot)$, which implies an analogous inequality for the respective covering numbers of the same sequential family $\cG_{J}\subset \cL^{\otimes_{J}}_{S}$. 

It  is also important for applications to describe the stability of covering numbers with respect to some elementary operations between families of functions, see for instance \cite[Lemmas 6.3, 6.4 and 6.5]{gyor:kohl:krzy:walk:02}, \cite[Section 5]{P90}, and \cite[Theorem 2.6.9]{VW96}. These translate to analogous ``stability properties'' for uniform entropy estimates.

Let us now give some instances of the notion of entropy estimates.
  
\begin{exam}[VC dimension. The ``Sauer-Shelah'' estimate.]
\label{exasaushe}
One important instance of uniform entropy estimates, which is part of the framework used in \cite{bargob}, is the following: for a function $f:\R^{d}\to \R$, define the {\it subgraph of $f$} as the set
\begin{align}
G_{f}^{+}:=\{(x,y)\in \R^{d}\times \R: y\leq f(x)\}.
\end{align}
The {\it VC-dimension $V_{\cF}$ of a family $\cF$ of functions $\R^{d}\to \R$}  {is the supremum of the natural numbers $l$ with the following property: there exists a set $G\subset \R^{d}\times \R$ with $l$ elements such that every subset $G'\subset G$ can be written in the form $G'=G\cap G_{f}^{+}$ for some $f\in \cF$}.  

When $\cF$ is a family of bounded, nonnegative functions $f:\R^{d}\to [0,B]$, one has the following uniform $L^{1}-$entropy estimate (\cite[Lemma 9.2 and Theorem 9.4.]{gyor:kohl:krzy:walk:02}) for the ``diagonal'' family 
\begin{align}
\label{equdiafam}
\cF_{1:n}:=\{\underbrace{(f,\dots,f)}_\text{$n$ times}: f\in \cF\}
\end{align}
 of hypotheses on $\cF$,  which we typically identify with $\cF$ itself:\footnote{Covering numbers for non-diagonal families are nonetheless implicit within what follows, for instance in the arguments behind \eqref{equsumargsec32indcas}.} for $r\in [0,B/4]$,  every $J\subset \{1,\dots,n\}$, and every $z_{J}\in (\R^{d})^{J}$,
\begin{align}
\log (\cN_{1}(r,\cF_{J},z_{J}))& \leq \lambda_{V_{\cF},B}(r)\\
& :=\log 3+V_{\cF}(1+\log 2 +\log ({B}/{r}) +\log(1+\log 3 +\log ({B}/{r}))),\\
\label{equdefunientestfinvc}
\end{align} 
which is clearly $O(\log(1/r))$ as $r\to 0^{+}$ when $V_{\cF}<\infty$\footnote{Note also that the restriction  $r\in [0,B/4]$ can be easily bypassed: one can for instance take $\lambda_{V_{\cF},B}(r)=0$ if $r>B$, and for $B\in [B/4,B]$, one can take $\lambda_{V_{\cF},B}(r):=\lambda_{V_{\cF},4B}(r)$, where $\lambda_{V_{\cF},4B}(r)$ is defined as in \eqref{equdefunientestfinvc} (valid for $r\in [0,4B/4]=[0,B]$). A similar trick allows us to give uniform entropy estimates via \eqref{equdefunientestfinvc} on (perhaps   nonpositive)  families $\cF$ of functions $f:\R^{d}\to [-B,B]$: the family $\cF'=\cF+B:=\{f+B:f\in \cF\}$ has the same covering numbers as $\cF$, satisfies $V_{\cF}=V_{\cF'}$, and its elements are functions $f:\R^{d}\to [0,2B]$.}, in particular when $\cF=T_{B}{\cH}$ is the family of truncated functions (see Section \ref{secleasqunot}) from a  vector space of dimension $d_{\cH}<\infty$, thanks to the bounds 
\begin{align}
V_{T_{B}\cH}\leq V_{\cH}\leq d_{\cH}+1
\end{align}
 (\cite[Theorem  9.5 (and previous paragraph) and Equation (10.23)]{gyor:kohl:krzy:walk:02}). 
 \end{exam}
 
 The  estimate \eqref{equdefunientestfinvc} is a consequence of the celebrated {\it Sauer-Shelah lemma} (\cite{Sau72},\cite{She72}). It  is therefore a relationship between the  {\it complexity of $\cF$}, as measured by $V_{\cF}$, and the notion of uniform entropy estimates.  
 
 There are other notions of complexity for families of functions, also associated to uniform entropy estimates, that are very relevant within the current literature, such as the (distribution--dependent) {\it Rademacher  complexity} and the  {\it fat shattering dimension}. See  \cite{mohros08} and    \cite{raksritew15} for respective discussions beyond the i.i.d. case.
 
\begin{exam}[Neural networks]
\label{exaneunet}
A second example is given by neural networks: it is  shown in  \cite[p.314]{gyor:kohl:krzy:walk:02} that if $\sigma:\R\to [0,1]$ is any cumulative distribution function (for instance a ``sigmoid'' function with asymptotes $y=0$ and $y=1$) and $\cF$ is the family of functions $f:\R^{d}\to \R$ of the form
\begin{align}
f(x)= b_{0}+ \sum_{k=1}^{N}b_{k}\sigma(u_{k}^{T}x+a_{k})
\end{align}
with $N\in \N$ fixed,  and with $((a_{k})_{k}, (b_{k})_{k}, (u_{k})_{k})\in  \R^{N}\times \R^{N+1} \times (\R^{d})^{N}$ subject to the restriction $\sum_{k}|b_{k}|\leq B $ for some $B>0$, then the corresponding diagonal family $\cF_{1:n}$ (see \eqref{equdiafam}) satisfies
\begin{align}
\log (\cN_{1}(r,\cF_{J},z_{J}))& \leq ((2d+5)N+1)(1+\log(12)+\log(B/r)+\log(N+1))
\label{equdefunientestfinvcneunet}
\end{align}  
for every $r\in (0,B/2)$.

Notice that the estimates in Examples \ref{exasaushe} and \ref{exaneunet}  do not depend on $|J|$. In our applications, the dependence on $|J|$ will be introduced by lower-bounding the radius $r\geq r(|J|)$ where these estimates are applied.
\end{exam}

\begin{remark}[Additional comments on $V_{\cF}<\infty$]
Restricting the analysis to the case $V_{\cF}<\infty$ is basically a convenience due the estimate \eqref{equdefunientestfinvc} for the quantitative bounds on the errors discussed in our applications, but one can extend these to some cases of ``infinite complexity'' ($V_{\cF}=\infty$) using similar estimates.
 
 One instance is the estimate \eqref{equdefunientestfinvcneunet} for neural networks (see \cite{son92} for examples showing that $V_{\cF}$ can be infinite within this context), but as indicated in \cite[Remarks 3.3, 3.5 and 3.19]{bargob},  one can extend the applications below to cases in which  the left--hand side of \eqref{equdefunientestfinvc} is bounded by a function of the form $O(({1}/{r})^{\alpha})$ ($r\to 0^{+}$) for some $\alpha\in (0,1)$.
 \end{remark}

\subsection{Weak least-squares error estimates under dependence}
In what follows, we provide some applications of the results above to distribution-free and nonparametric error bounds associated to schemes based in the method of least-squares regression.
\subsubsection{Least-squares setting}
\label{secleasqunot}
We recover the following definitions and conventions from \cite{bargob} (see Section \ref{secnot} for previous notation): 
\begin{itemize}
\item {\it Truncation Operator.} First, we remind the {\it truncation operator,} defined for a constant $B>0$ and associating to any real--valued function $g$  the function $T_{B}g$ defined as
\begin{align}
T_{B}g(x)=\max\{\min\{g(x),B\},-B\}.
\end{align} 

\item {\it Least--squares regression (LSR) objects.} Consider a  random vector  $(X,Y)_{1:\infty}$ of $(\R^{d}\times \R)^{\N}$, assume that  for all $k$, $Y_{k}\in L^{2}_{\P}$, and  pick a version  $\regk:\R^{d}\to \R$ of $\Esp{Y_{k}|X_{k}}$. Thus 
\begin{align}
\regk(X_{k})=\Esp{Y_{k}|X_{k}},\,\,\,\,\P-a.s., & \qquad k=1,\dots,n.
\end{align}
We also write $\reg:=(\regk)_{k=1}^{n}$.

For a fixed $n\in \N$, consider the data $D_{n}:=(X,Y)_{1:n}$. Given a family $\cF$ of Borel-measurable functions $\R^{d}\to \R$, let   $\hatreg=\hatreg(\cF,D_{n})$ be a solution (assume it exists) of the least--squares regression problem associated to $\cF$ and $D_{n}$:  if we identlfy 
\begin{align}
\label{equideftup}
{ f}\equiv\underbrace{(f,\dots,f)}_\text{$n$ times}
\end{align}
 for $f\in \cF$ (compare with \eqref{equdiafam}), then
\begin{align}
\hatreg\in \arg\min_{f\in\cF}A_{(X,Y)_{1:n}}|{ f}-{ y}_{1:n}|^{2},
\end{align} 
where ${ y}_{1:n}=(y,\dots,y)$ with $y:\R^{d}\times \R\to \R$ the projection on the second coordinate ($y(x_{0},y_{0})=y_{0}$), and where we naturally identify $f\in \cF$ with the function $\R^{d}\times \R\to \R$ whose value at $(x,y)$ is $f(x)$. Notice that, consistently with  \eqref{equideftup}
\begin{align}
{\hatreg}\equiv\underbrace{(\hatreg,\dots,\hatreg)}_\text{$n$ times}.
\end{align}
 \item {\it Pointwise deviations of the least--squares error.} In this context, we reserve a special notation for the family $\cG_{\cF,1:n}=\{\bgf\}_{f\in\cF}$ whose elements are the sequential functions 
\begin{align}
\label{equdefgfk}
\bgf:=|{ y}_{1:n}-{ f}|^{2}-|{ y}_{1:n}-\reg|^{2}.
\end{align}
From here, the meaning of $\cG_{\cF,J}$ for any $J\subset\{1,\dots,n\}$ is clear (see \eqref{equdeffjpri}).
\end{itemize}

\subsubsection{A weak $L^{2}-$error estimate for 
dependent samples}
We \nch{continue} with the following ``$\beta-$version'' of \cite[Theorem 3.1]{bargob}. The setting is that in Section \ref{secleasqunot}:

\begin{thm}
\label{theweaerrestbet}
{\upshape ($\beta-$version of \cite[Theorem 3.1]{bargob}).} Assume that  $\cF$ is a pointwise measurable class of functions  with associated $VC-$dimension $V_{\cF}<\infty$, and that $||Y_{k}||_{\P,\infty}\leq B$ 
for some $B>0$ and all $k$. Assume further that $(c,\lambda,n,m)\in (1,\infty)\times(1,\infty)\times \N\times \N$ are such that
\begin{align}
\label{equhyplamcdelthe31bargobbetmix}
\lambda\leq \frac{3+\sqrt{1+8c}}{4},& \qquad \left\lfloor\frac{n}{m}\right\rfloor\geq \exp\left(\frac{c^{2}-71}{4V_{\cF}}\right),
\end{align} 
(in particular $n\geq m$), then the estimate
\begin{align}
\Esp{\mmxon|T_{B} \hatreg-\reg|^{2}
} &\leq 
  \frac{B^{2}}{\left\lfloor\frac{n}{m}\right\rfloor}\theta_{0}\left(1+\theta_{1}+V_{\cF}(\theta_{2}+\log(\theta_{2}))\right)
\\
&+16B^{2}(1+\lambda) n\beta_{(X,Y)_{1:\infty}}(m)+\lambda \inf_{f\in\cF}
\mmxon | f-\reg
|^{2}\\
&(=\underbrace{\mbox{\upshape``Variance'' $+$ ``$\beta-$mixing error'' }}_{\mbox{\upshape ``Statistical error''}} \mbox{\upshape$+$ ``scaled bias''.})\\
\label{gyothe115equ}
\end{align}
holds, where
\begin{align}
\theta_{0}=&\theta_{0}(\lambda,c):=32\left(\frac{1}{3}(1-\frac{1}{c})(1-\frac{1}{\lambda})+(2\lambda-1)\right)^{2}(\frac{c}{c-1})^{3}\frac{\lambda}{\lambda-1},\\
\theta_{1}=& \theta_{1}(c,m):=\log(6(c+1)(2c+3))+\log m, \\
\theta_{2}=& \theta_{2}(c,n,m):= 1+\log 24+\log(1+\sqrt{1+\frac{c(c+1)}{\left\lfloor\frac{n}{m}\right\rfloor+1}})-\log(c-\frac{1}{c})+\log (\left\lfloor\frac{n}{m}\right\rfloor+1).\end{align}
\end{thm}

\begin{remark}[A simplified version of the variance in \eqref{gyothe115equ}]
It is easy to see that for every $c',\lambda'>1$, there exists a constant 
$C_{c',\lambda'}>0$ such that
\begin{align}
\label{equbouvarcoef}
\frac{1}{C_{c',\lambda'}}\frac{V_{\cF}}{\lambda-1}(1+\log n-\log m)\leq &\theta_{0}\left(1+\theta_{1}+V_{\cF}(\theta_{2}+\log(\theta_{2}))\right) \\ \leq & C_{c',\lambda'}\frac{V_{\cF}}{\lambda-1}(\log c +\log n),
\end{align}
provided that $(n,c,\lambda)\in \N\setminus\{1\}\times (c',\infty)\times (1,\lambda')$. We shall when convenient write  the variance term in \eqref{gyothe115equ} as
\begin{align}
\label{equsimvar}
O\left(\frac{B^{2}_{n}V_{\cF_{n}}(\log c_{n} +\log n)}{(\lambda_{n}-1){n}}{m_{n}}\right),
\end{align}
with the ``right'' of letting $n\to\infty$ as far as $\{\lambda_{n}\}_{n}\subset (1,\infty)$ is bounded and $\{c_{n}\}_{n}\subset (1,\infty)$ is away from 1, and provided that \eqref{equhyplamcdelthe31bargobbetmix} holds for the parameters $(c_{n},\lambda_{n},V_{\cF_{n}},m_{n})$.
\end{remark}
\begin{remark}[The ``$\beta-$mixing error--variance'' tradeoff]
 Notice also that reducing $m$ simultaneaously increases the $\beta-$mixing error and reduces the variance  in \eqref{gyothe115equ}. Our choice of $m$ in the applications below is based on a qualitatively optimal tradeoff between these errors:  the tradeoff is made for $m=m_{n}$ with the goal of  minimizing the distance to the smallest possible statistical error, achieved in the independent case in which $\beta_{(X,Y)_{1:\infty}}(1)=0$ and the statistical error is therefore equal to the variance term in \eqref{gyothe115equ} for $m=1$. 
 \end{remark}

{\it Proof of Theorem \ref{theweaerrestbet}}.
First, as proved in \cite[Section 3.2]{bargob}, we have the estimate
\begin{align}
\label{ineprothe38exp}
\Esp{\mmxon|T_{B} \hatreg-\reg|^{2}}\leq \Esp{\left(\sup_{f\in T_{B}\cF}(\mmxyon-\lambda \emxyon)\bgfon\right)^{+}}+\lambda\inf_{f\in \cF}\mmxon|{ f}-\reg|^{2}
\end{align}

We proceed now to bound conveniently the distribution function $[0,\infty)\to [0,1]$  defined by
\begin{align}
t\mapsto \Prob{\sup_{f\in T_{B}\cF}(\mmxyon-\lambda \emxyon)\bgfon\geq t}.
\end{align} 

Assuming that $B=1/4$, which gives that $|g_{f}^{k}(x)|\leq 1$ for all $k$ and $x$, the arguments in \cite[Section 3.2.]{bargob} lead to the   inequalities  

\begin{align}
\Prob{\sup_{f\in T_{1/4}\cF}(\mmxyJsta-\lambda \emxyJsta)\bgfJ\geq t}&\\
\leq 3G_{0}(c)\Esp{\cN_{1}(G_{1}(c,\lambda)t_{0}(c,\lambda,|J|), T_{1/4}\cF,X_{1:n})}\exp(-b({c,\lambda})|J| t)&\\
 \leq 3G_{0}(c)\left(\frac{e}{G_{1}(c,\lambda)t_{0}(c,\lambda,|J|)}\log(\frac{3e}{2G_{1}(c,\lambda)t_{0}(c,\lambda,|J|)})\right)^{V_{\cF}}\exp(-b({c,\lambda})|J| t) &\\
=:  a_{0}(c,\lambda,|J|) \exp(-b({c,\lambda})|J| t)\\
\label{equsumargsec32indcas}
\end{align}
for every $J\subset \{1,\dots,n\}$ and every random element $(X,Y)_{J}^{*}$ of $(\R^{d}\times\R)^{J}$ with independent entries and the same marginals as $(X,Y)_{J}$, with $G_{0}, G_{1}$, and $b$ given by
\begin{align}
G_{0}(c)&:=2(c+1)(2c+3),\qquad
G_{1}(c,\lambda):=\frac{1}{8}\frac{1}{\lambda(c-1)+1}(1-\frac{1}{c}),\\
b({c,\lambda})&:=\frac{1}{2}\frac{1}{(\frac{1}{3}(1-\frac{1}{c})+(2\lambda-1)\frac{\lambda}{\lambda-1})^{2}}(1-\frac{1}{c})^{{3}}\frac{\lambda}{\lambda-1},
\label{equg0g1bpro}
\end{align}
 and provided that 
\begin{align}
t\geq t_{0}(c,\lambda,|J|):=\frac{-(\lambda-1)+\sqrt{(\lambda-1)^{2}+c(c+1)\lambda^{2}/|J|}}{2}.
\end{align}
Therefore we have, for every $J\subset \{1,\dots,n\}$ 
and every $(X,Y)_{J}^{*}$ as indicated, the estimate
\begin{align}
\Prob{\sup_{f\in \cF}(\mmxyJsta-\lambda \emxyJsta)\bgfJ\geq t}\leq &\1_{\{t< t_{0}(|J|,c,\lambda)\}}+ a_{0}(c,\lambda,|J|
) \exp(-b({c,\lambda})|J| t) \1_{\{t_{0}(|J|,c,\lambda)\leq t\}}.
\end{align}
This gives rise, via Proposition \ref{prop_corsemabsexp} and elementary estimates, to the inequality
\begin{align}
\Prob{\sup_{f\in \cF}(\mmxyon-\lambda \emxyon)\bgfon\geq t}
\leq (n\beta_{_{(X,Y)_{1:\infty}}}(m) + L_{c,\lambda}(n,m,t))\1_{\{t \leq {(1+\lambda)}\}},\\\label{equestdisfunthe38}
\end{align}
where
\begin{align}
L_{c,\lambda}(n,m,t):=\1_{\{t< t_{0}(\left\lfloor\frac{n}{m}\right\rfloor,c,\lambda)\}} + m\, a_{0}(c,\lambda,\nch{\left\lfloor\frac{n}{m}\right\rfloor+1}
) \exp(-b({c,\lambda})\left\lfloor\frac{n}{m}\right\rfloor t)) )\1_{ \{t_{0}(\left\lfloor\frac{n}{m}\right\rfloor,c,\lambda)\leq t.\}}
\end{align}
The desired estimate for the case $B=1/4$ follows from \eqref{ineprothe38exp} and integration with respect to $t$ (and Lebesgue measure) of the right--hand side of \eqref{equestdisfunthe38}, with the integral of $L_{c,\lambda}(n,m,\cdot)$ estimated as in the arguments following \cite[Equation (3.13)]{bargob}. The estimate for general $B>0$ follows by an homogenization argument (see the homogenization argument after \cite[Equation (3.19)]{bargob}).
\qed

\medskip
\subsubsection{Weak rates for $\beta-$mixing schemes}
It is worth discussing what the right--hand side of  \eqref{gyothe115equ} says about weak consistency,  and to introduce some cases and consequences of special importance which fall under this discussion. The setting is again that in Section \ref{secleasqunot}.

 As a first consequence, we point out the following result:
\begin{prop}[Weak rate for uniformly bounded schemes]
\label{corweaconbou}
Assume that  $\cF$ is a pointwise measurable family  with associated $VC-$dimension $V_{\cF}<\infty$ and with 
\begin{align}
\sup_{f,k}\{ ||f(X_{k})||_{\P,\infty}, ||Y_{k}||_{\P,\infty} \}\leq B
\end{align}
for some $B\in(0,\infty)$, then for any sequence $(m_{n})_{n}$ of natural numbers 
and any bounded positive sequence $(\delta_{n})_{n}$,
\begin{align}
 \Esp{\mmxon|\hatreg-\reg|^{2}-\inf_{f\in\cF}\mmxon|f-\reg|^{2}}\\
 =O\left(\frac{ \log n}{\delta_{n}{n}}{m_{n}}+{n}\beta_{(X,Y)_{1:\infty}}(m_{n})+\delta_{n}\left({n}\beta_{(X,Y)_{1:\infty}}(m_{n})+\inf_{f\in\cF}\mmxon|f-\reg|^{2}\right)\right).\\
  \label{equweaconallbou}
 \end{align}
\end{prop}
\begin{proof}
This is an immediate cosequence of \eqref{gyothe115equ} and \eqref{equsimvar}, by choosing  $B_{n}=B, V_{\cF_{n}}=V_{\cF}$, $\lambda_{n}=1+\delta_{n}$ and (say) $c_{n}=2$. 
\end{proof}

 \medskip

 \begin{remark}[Some consequences of \eqref{equweaconallbou}]
 It follows in particular that, under the hypotheses of Proposition \ref{corweaconbou}, the left-hand side of \eqref{equweaconallbou}  converges to zero if there exists a sequence  $(m_{n})_{n}$ such that
\begin{align}
\label{hypgenbetmixweacon}
m_{n}\frac{\log n}{n}+n\beta_{(X,Y)_{1:\infty}}(m_{n})\to_{n} 0
\end{align}
(take $\delta_{n}:=(m_{n}\log n /n)^{1/2}$). Notice also that the rate at the right--hand side  of \eqref{equweaconallbou} admits convenient interpretations in interesting cases: if for instance $(X,Y)_{1:\infty}$ is $m-$dependent (see item \ref{itemdepsuf}. in Section \ref{prop:remrbetzjinwor}) and conditionally stationary in the sense that  for some $\Phi:\R^{d}\to \R$, $\Phi(X_{k})=\Esp{Y_{k}|X_{k}}$, $\P-$a.s.,  and if $\Phi\in \cF$ (unbiased case), \eqref{equweaconallbou} gives the rate of convergence $O(\log n/n)$ to zero for the expected squared error of the least--squares loss $\Esp{\int_{\R^{d}}|\hatreg(x)-\Phi(x)|^{2}\, \dx}$ (take $m_{n}=m+1$ and $\delta_{n}=1$ in \eqref{equweaconallbou}).
\end{remark}

In any case \eqref{hypgenbetmixweacon} requires that 
\begin{align}
\label{equbetmixo1oven}
\beta_{(X,Y)_{1:\infty}}(m_{n})=o(n^{-1}),
\end{align}
 for some sequence $(m_{n})_{n}$ satisfying $m_{n}=o(n/\log n)$. It is necessary for \eqref{hypgenbetmixweacon} that
 \begin{align}
 \label{equdefquabetmixfas}
 \lim_{n} n\,\beta_{(X,Y)_{1:\infty}}(n)= 0 
 \end{align}
 because $(\beta_{(X,Y)_{1:\infty}}({n}))_{n}$ is decreasing (see Remark \ref{prop:remrbetzjinwor}
 ), and in particular that 
 \begin{align}
 \label{equdefquabetmix}
 \lim_{n} \beta_{(X,Y)_{1:\infty}}(n)= 0.
 \end{align}

 \begin{remark}[The ``$\beta-$mixing'' assumption]
 \label{rembetmixass}
 Notice that, in general, \eqref{equdefquabetmix} is  less restrictive\footnote{See the footnote  on the definition of $\beta_{Z_{\cdot}}$ in page \pageref{foonotbetnotbetmix}.} than {\it the $\beta-$mixing assumption} on $(X,Y)_{1:\infty}$, which  amounts to the hypothesis \begin{align}
 \label{equdefbetmix}
 \lim_{m}\sup_{k}\beta(\sigma((X,Y)_{1:k}),\sigma((X,Y)_{k+m:\infty}))= 0.
 \end{align}
However,  as we pointed out after Definition \ref{defsubpolmix}, { \eqref{equdefquabetmix} is exactly the beta-mixing assumption \eqref{equdefbetmix} when $(X,Y)_{1:\infty}$ is a Markov process. }
\end{remark}

For the rates in Definitions \ref{defsubbetmix} and \ref{defsubpolmix}, we deduce the following versions of Theorem \ref{theweaerrestbet}: 

\begin{prop}[Weak rate of convergence for subexponentially $\beta-$mixing samples]
\label{corwearatsubexpmix}
There exists a universal constant $C$ with the following property: 
if $(X,Y)_{1:\infty}$ is subexponentially $\beta-$mixing (Definition \ref{defsubbetmix}) with parameters $(a,b,\gamma)$, and if for some $B\in(0,\infty)$ and all $k\in \{1,\dots,n\}$, $||Y_{k}||_{\P,\infty}\leq B$, then for any 
\begin{align}
\label{equhyplamcdelthe31bargobbetmixexp}
1<\lambda\leq \frac{3+\sqrt{1+8\sqrt{71}}}{4},
\end{align}
 the statistical error in \eqref{gyothe115equ} is bounded by 

\begin{align}
\frac{C}{n}\left(\frac{B^{2}V_{\cF}}{(\lambda-1)}{(1+ \log n)}+a\right)\left(\frac{2\log n}{b}\right)^{1/\gamma}
\end{align}
provided that
\begin{align}
\label{equresnweaconthe}
1\leq \left(\frac{2\log n}{b}\right)^{1/\gamma}\leq \frac{n}{2}.
\end{align}
\end{prop}

\begin{proof}
For any positive real number $\alpha$ and any $x\geq 2$, the inequality
\begin{align}
\alpha x +  a{n} \exp(-b\left\lfloor x\right\rfloor^{\gamma})\leq  \alpha x + 
a{n}
\exp(-\frac{b}{2^{\gamma}} x^{\gamma})
\end{align}
holds. It follows  that, if $C_{c',\lambda'}$ is the constant from \eqref{equbouvarcoef} corresponding to $(c',\lambda')=(\sqrt{71},{3+\sqrt{1+8\sqrt{71}}}/{4})$ and
\begin{align}
\alpha_{\lambda,n}:=2C_{c',\lambda'}\frac{B^{2}V_{\cF}(\log \sqrt{71} + \log n)}{(\lambda-1)n}
\end{align}
then, under the subexponentially mixing hypothesis \eqref{equdefexpmix}, the statistical error in  \eqref{gyothe115equ} is bounded by 
\begin{align}
\min_{x\in [2,n]}\{\alpha_{\lambda,n} x + 
a{n}
\exp(-\frac{b}{2^{\gamma}} x^{\gamma})\}
\end{align}

Taking $x:={2^{1+1/\gamma}}\left({\log n}/{b}\right)^{1/\gamma}$, which lies in $[2,n]$ in virtue of \eqref{equresnweaconthe}, we get the bound
\begin{align}
{2^{1+1/\gamma}}\left(\frac{\log n}{b}\right)^{1/\gamma}\alpha_{\lambda,n}+ \frac{a}{n}.
\end{align}
for the statistical error in  \eqref{gyothe115equ}. The result follows from an easy estimation on this bound.
\end{proof}

A similar (and easier) argument, taking this time $x:=\left\lceil n^{^{2/\gamma+1}}\right\rceil$ and estimating via \eqref{equdefpolmix}, gives the corresponding weak rate for subpolinomially $\beta-$mixing samples:
\begin{prop}[Weak rate of convergence for subpolinomially $\beta-$mixing samples]
\label{corwearatsubpolmix}
There exists a universal constant $C$ with the following property: 
if $(X,Y)_{1:\infty}$ is subpolinomially $\beta-$mixing (Definition \ref{defsubpolmix}) with parameters $(a,\gamma)$, and if for some $B\in(0,\infty)$ and all $k\in \{1,\dots,n\}$, $||Y_{k}||_{\P,\infty}\leq B$, then for any 
\begin{align}
\label{equhyplamcdelthe31bargobbetmixpol}
1<\lambda\leq \frac{3+\sqrt{1+8\sqrt{71}}}{4},
\end{align}
 the statistical error in \eqref{gyothe115equ} is bounded by 
\begin{align}
\frac{C}{ n^{^{(\gamma-1)/(\gamma+1)}}}\left(\frac{B^{2}V_{\cF}}{(\lambda-1)}{(1+ \log n)}+a\right) .
\end{align}
\end{prop} 
Notice that, in Corollaries \ref{corwearatsubexpmix} and \ref{corwearatsubpolmix}, we recover the rates of the independent case by letting $\gamma\to\infty$.

\bibliographystyle{alpha}
\bibliography{notesonnonparametricregressiondepcas}
\end{document}